\documentclass[hidelinks, 12pt]{amsart}

\usepackage{amsmath}
\usepackage[psamsfonts]{amssymb}
\usepackage{amsfonts}
\usepackage{algorithm2e}
\usepackage{graphicx}
\usepackage{verbatim}
\usepackage{dsfont}
\usepackage{tikz-cd}
\usepackage{mathrsfs}
\usepackage{color}
\usepackage[english]{babel}
\usepackage{fontenc}
\usepackage{indentfirst}
\usepackage{tikz}
\usetikzlibrary{matrix,arrows,decorations.pathmorphing}
\usepackage{algorithm2e}
\usepackage[all]{xy}
\usepackage[T1]{fontenc}
\usepackage{hyperref}
\usepackage{mathtools}
\usepackage{multicol}

\newtheorem{theorem}{Theorem}[section]
\newtheorem{prop}[theorem]{Proposition}
\newtheorem{lemma}[theorem]{Lemma}
\newtheorem{cor}[theorem]{Corollary}

\newtheorem{ex}[theorem]{Example}
\newtheorem{dfn}[theorem]{Definition}
\newtheorem{remark}[theorem]{Remark}
\newtheorem{claim}[theorem]{Claim}

\def\ep{\epsilon}

\def\R{\mathbb{R}}
\def\Z{\mathbb{Z}}
\def\N{\mathbb{N}}

\def\C{\mathbb{C}}

\def\T{\mathbb T}

\def\cal R{\mathcal R}
\def\P{\mathcal P}
\def\eps{\epsilon}

\DeclarePairedDelimiter\floor{\lfloor}{\rfloor}

\begin{document}

\title{Hamiltonian knottedness and lifting paths from the shape invariant}
\date{\today}

\author{Richard Hind}
\email{hind.1@nd.edu}
\address{Department of Mathematics, University of Notre Dame, Notre Dame, IN 46556, USA}

\author{Jun Zhang}
\email{jun.zhang.3@umontreal.ca}
\address{Centre de Recherches Math\'ematiques, University of Montreal, C.P. 6128 Succ. Centre-Ville Montreal, QC H3C 3J7, Canada}

\maketitle 

\begin{abstract} The Hamiltonian shape invariant of a domain $X \subset \R^4$, as a subset of $\R^2$, describes the product Lagrangian tori which may be embedded in $X$. We provide necessary and sufficient conditions to determine whether or not a path in the shape invariant can lift, that is, be realized as a smooth family of embedded Lagrangian tori, when $X$ is a basic $4$-dimensional toric domain such as a ball $B^4(R)$, an ellipsoid $E(a,b)$ with $\frac{b}{a} \in \N_{\geq 2}$, or a polydisk $P(c,d)$. As applications, via the path lifting, we can detect knotted embeddings of product Lagrangian tori in many toric $X$. 
We also obtain novel obstructions to symplectic embeddings between domains that are more general than toric concave or toric convex.

\tableofcontents

\end{abstract}

\section{Notation}

Here we gather some common notations. We work in $\R^4 \equiv \C^2$ with the standard symplectic form $\omega = \frac{i}{2} \sum_{k=1}^2 dz_k \wedge d\overline{z}_k$. The moment map is
\[ \mu:\C^2 \to \R_{\ge 0}^2 \,\,\,\,\, (z_1, z_2) \mapsto (\pi|z_1|^2, \pi|z_2|^2).\]
We use coordinates $(r, s)$ on $\R_{\ge 0}^2$.
Given a subset $\Omega \subset \R_{\ge 0}^2$ we define the corresponding toric domain $X_{\Omega} := \mu^{-1}(\Omega) \subset \C^2$. A toric {\it star-shaped} domain $X_{\Omega}$ has $\partial \Omega$ transversal to the radial vector field $X_{\rm rad} = r\frac{\partial}{\partial r} + s \frac{\partial}{\partial s}$ of $\R_{>0}^2$. In the case when this subset $\Omega$ is of the form 
\[\Omega = \{(r,s) \in \R^2 \,| \, 0 \le r \le a, 0 \le s \le f(r) \}, \]
we say that $X_{\Omega}$ is (toric) {\it convex} if $f$ is concave, and $X_{\Omega}$ is (toric) {\it concave} if $f$ is convex. Examples of toric domains which are both convex and concave are symplectic (closed) ellipsoids. We define the (open) ellipsoid $E(a,b) := X_{\Delta(a,b)}$ with $0<a\leq b$, where $\Delta(a,b) := \{0 \le r < a, 0 \le s < b - \frac{br}{a} \}$. The ball of capacity $R$ is denoted by $B^4(R) = E(R,R)$. We say that an ellipsoid $E(a,b)$ is {\it integral} if $\frac{b}{a} \in \N_{\geq 2}$. The polydisk $P(c,d): = X_{\Box(c,d)}$ with $0<c \leq d$,  where $\Box(a,b) := \{0 \le r <a, 0 \le s <b\}$. A product Lagrangian torus is a torus 
\begin{equation} \label{prod-torus}
L(r,s) := X_{\{ (r,s) \}} = \mu^{-1}(r,s).
\end{equation}
Given subsets $U, V \subset \R^4$ we write $\Phi: U \hookrightarrow V$ to mean that there exists a Hamiltonian diffeomorphism $\Phi$ of $\R^4$ embedding $U$ into ${\rm int}(V)$. If $\Phi$ is not emphasized, we simply denote such an embedding by $U \hookrightarrow V$.

We will denote by ${\mathcal L}(X)$ the set of Lagrangian tori in a domain $X \subset \C^2$ which are Hamiltonian isotopic in $\C^2$ to a product torus. The space ${\mathcal L}(X)$ is equipped with the smooth topology.


\section{Introduction}

A fundamental problem in symplectic topology is to understand the space of Lagrangian submanifolds, denoted by ${\mathcal L}(X)$, of a given manifold $X$, and in particular the action of the Hamiltonian group ${\rm Ham}(X)$ on $\mathcal L(X)$. We will describe some steps in this direction when the symplectic manifold is a domain in $\R^4$, including balls, integral ellipsoids, and polydisks. Our Lagrangian submanifolds are tori, and for simplicity we will restrict attention to those which are Hamiltonian isotopic to a product torus $L(r,s)$ defined in (\ref{prod-torus}). Recall that when $X = \R^4$ the only known Lagrangian tori in $\R^4$ which do not fall into this category are embedded tori that Hamiltonian isotopic to scalings of the Chekanov torus, see \cite{CS10}. 

The Lagrangian tori in a given domain $X$ are described up to Hamiltonian diffeomorphism in $\R^4$ by the (Hamiltonian) shape invariant
\begin{align}\label{dfn-hsi} 
{\rm Sh}_{H}(X) & : = \left\{(r, s) \in \R_{>0}^2 \, \big| \, L(r, s) \hookrightarrow X \right\}.
\end{align}
The study of the shape invariant was initiated by Eliashberg in \cite{Eli91}. Note that ${\rm Sh}_{H}(X)$ contains strictly more information than the possible area classes of embedded Lagrangian tori (which we simply called the shape invariant in \cite{hindzhang}). Indeed the product tori $L(1,2)$ and $L(2,3)$ have integral Maslov $2$ bases with the same area classes, but as stated in Theorem \ref{shapecalc} below, $L(1,2) \hookrightarrow B(3+\ep)$ for any arbitrarily small $\ep>0$ while there is no such embedding from $L(2,3)$.

It is often convenient to work with the reduced (Hamiltonian) shape invariant denoted by ${\rm Sh}_{H}^+(X) := {\rm Sh}_{H}(X) \cap {\{r \leq s\}}$. As examples, the Hamiltonian shape invariants of balls $B^4(R)$ and polydisks $P(c,d)$ were worked out by the first author and Opshtein in \cite{HO19}, and the current authors computed the shape invariant of integral ellipsoids in \cite{hindzhang}.

\begin{theorem}[\cite{HO19, hindzhang}] \label{shapecalc} We have the following computations of the reduced (Hamiltonian) shape invariants. 
\begin{itemize}
\item[(i)] When $X = B^4(R)$, 
\[ {\rm Sh}_{H}^+(B^4(R)) = \left\{ (r,s) \in \R_{>0}^2 \, \bigg| \, r+s <R \,\,\, \mbox{\rm or} \,\,\, r < \frac{a}{2} \right\} \cap \{r \leq s\}. \]
\item[(ii)] When $X = E(a,b)$ with $\frac{b}{a} \in \N_{\geq 2}$, 
\[ {\rm Sh}_{H}^+(E(a,b)) =\left\{ (r,s) \in \R_{>0}^2 \, \bigg| \, \frac{r}{a} + \frac{s}{b} <1 \,\,\, \mbox{\rm or} \,\,\, r < \frac{a}{3} \right\} \cap \{r \leq s\}. \]
\item[(iii)] When $X = P(c,d)$ with $0< c \leq d$, 
\[ {\rm Sh}_{H}^+(P(c,d)) = \left\{ (r,s) \in \R_{>0}^2 \, \bigg| \, \begin{array}{l} r< c\\ s<d \end{array} \,\,\, \mbox{\rm or} \,\,\, r < \frac{c}{2}\right\}\cap \{r \leq s\}. \]
\end{itemize}
\end{theorem}

\begin{remark} (i) Note that the subsets ${\rm Sh}_H^+$ of the basic toric domains in Theorem \ref{shapecalc} are all formed, modulo the intersection with $\{r \leq s\}$, by the moment image $\mu(X)$ plus a vertical long strip.

(ii) The proof of statement (ii) in Theorem \ref{shapecalc} utilized some results from K. Siegel \cite{Sieip} which are yet to appear; however the results in the current paper are independent of Theorem \ref{shapecalc}.\end{remark}

\subsection{Hamiltonian knottedness} \label{ssec-ham-knotted} There exists a Hamiltonian diffeomorphism of $\R^4$ mapping $L(r,s)$ onto $L(r', s')$ if and only if $\{r,s\} = \{r', s'\}$, see \cite{Che96}. Hence there is a well defined projection map 
\begin{equation} \label{area-class-map}
\P: {\mathcal L}(X) \to {\rm Sh}_{H}^+(X), \,\,\,\,\mbox{given by} \,\, L \simeq L(r,s) \mapsto \, (r,s).
\end{equation}
This map is continuous. Indeed, by Weinstein's Neighborhood Theorem, a sequence of Lagrangians $L_n \to L$ in $\mathcal L(X)$ can be thought of sections of $T^* L$ or, up to a Hamiltonian diffeomorphism, as sections of the normal bundle $T^* L(r,s)$ of a product torus, where $\P(L)=(r,s)$. Such sections are Hamiltonian isotopic to constant sections, which correspond to $L(r_n,s_n)$, for $(r_n, s_n)$ converging to $(r,s)$.

The Hamiltonian diffeomorphism group ${\rm Ham}(X)$ acts on ${\mathcal L}(X)$ and preserves the fibers of $\P$. Paths in a fiber correspond to Lagrangian isotopies with a fixed area class $(r,s) \in {\rm Sh}_{H}^+(X)$, and these are realized by a Hamiltonian isotopy in $X$, see Theorem 0.4.2 in \cite{chaperon83}. Hence, path connected components of the fibers are precisely the orbits of ${\rm Ham}(X)$, and if a fiber over a point $(r,s)$ happens to be disconnected then we have embeddings $L(r,s) \hookrightarrow X$ which are not Hamiltonian isotopic in $X$. This motives the following definition. 

\begin{dfn} \label{dfn-knotted} Suppose the product Lagrangian torus $L(r,s)$ embeds into $X$ by inclusion, i.e., $(r,s) \in \mu(X)$. Then we call an embedded Lagrangian torus in $X$ {\rm unknotted} if it is in the same component as $L(r,s)$, and {\rm knotted} if it lies in other components. \end{dfn} 

More explicitly, if $L(r,s) \subset X$, then an embedded Lagrangian torus $L \in \P^{-1}((r,s))$ is unknotted if there exists a Hamiltonian isotopy {\bf in} ${\boldsymbol X}$, denoted by $\{\Phi_t\}_{t \in [0,1]}$, such that $\Phi_0 = \mathds{1}_X$ and $\Phi_1(L) = L(r,s)$. Note that, for $(r,s)$ with $r \neq s$, even though all Lagrangian tori in the fiber $\P^{-1}((r,s))$ are conjectured to be Hamiltonian isotopic in $\R^4$, they are not necessarily Hamiltonian isotopic in $X$. In other words, the action of ${\rm Ham}(X)$ on a fiber of $\P$ may not be transitive.  

\begin{remark} The (un)knottedness defined in Definition \ref{dfn-knotted} is identical to that for  
symplectic embeddings discussed in \cite{McD91,McD09,C-G19,GU19}. 
\end{remark}

Our first result says that for some basic domains many fibers of the projection $\P$ are indeed disconnected, so knotted Lagrangian tori are quite common. The following results will be proved in subsection \ref{ssec-thm-knotted-ball-ellipsoid}.

\begin{theorem} \label{knotted-ball}
Let $X=B^4(R)$. Then for any area classes 
\begin{equation} \label{knotted-ball-1} 
(r,s) \in \Delta(R,R) \cap \left\{(r,s) \in {\rm Sh}_H^{+}(X)\, |\, 3r \le R\,\, \mbox{and} \,\,2r + s > R\right\},
\end{equation}
there exist knotted Lagrangian tori in the fiber $\P^{-1}((r,s))$. 
\end{theorem}

\begin{theorem} \label{knotted-ellipsoid}
Let $X=E(a,b)$ with $k: = \frac{b}{a} \in \N_{\ge 2}$. Then for any area classes 
\begin{equation} \label{knotted-ellipsoid-1}
(r,s) \in \Delta(a,b) \cap \left\{(r,s) \in {\rm Sh}_H^{+}(X)\, |\, 2r \le a\,\, \mbox{and} \,\,(k+1)r + s > b\right\},
\end{equation}
there exist knotted Lagrangian tori in the fiber $\P^{-1}((r,s))$. 
\end{theorem}

\begin{theorem} \label{knotted-polydisk}
Let $X= P(c,d)$. Then for any area classes 
\begin{equation} \label{knotted-polydisk-1}
(r,s) \in \Box(c,d) \cap \left\{(r,s) \in {\rm Sh}_H^{+}(X)\, |\, 2r \le c\,\, \mbox{and} \,\,r + s > d\right\},
\end{equation} 
there exist knotted Lagrangian tori in the fiber $\P^{-1}((r,s))$. 
\end{theorem}

Detecting Hamiltonian knotted Lagrangian tori is also closely related to the symplectic embeddings, and sometimes we can detect more knotted Lagrangian tori in both $E(a,b)$ and $B^4(R)$ with the help of symplectic embeddings. We will discuss this in detail in subsection \ref{ssec-symp-emb}. 

\subsection{Path lifting} \label{ssec-path-lifting} The main results in subsection \ref{ssec-ham-knotted} are consequences of a novel analysis of the path lifting problem of the projection $\P$. To start, let us give the following key definition. 

\begin{dfn} \label{dfn-path-lift}
A smooth path $\gamma: [0,T] \to {\rm Sh}_{H}^+(X)$ where $\gamma(0)=(r_0, s_0)$ {\rm lifts to ${\mathcal L}(X)$} if there exists a smooth family of embedded Lagrangian tori in $X$, denoted by $\{L_t\}_{t \in [0,T]}$, with $\P(L_t) = \gamma(t)$ and $L_0 = L(r_0, s_0)$.\end{dfn} 

In other words, unless stated otherwise, we will always assume our lifts start from an inclusion, and so a necessary condition for a lift is that $\gamma(0) \in \mu(X)$. Also, denote by $\mu(X)^+ : = \mu(X) \cap \{r\leq s\}$.

\begin{ex} \label{image-path}  \normalfont For any smooth path $\gamma: [0,T] \to \mu(X)^+ \subset {\rm Sh}_{H}^+(X)$, it lifts to ${\mathcal L}(X)$ since one can consider the family of product Lagrangian tori $\{L(r_t, s_t)\}_{t \in [0,T]}$ with area classes smoothly changing along $\gamma$. \end{ex}

We call any path in Example \ref{image-path} a Type-I path, and any other path in ${\rm Sh}_{H}^+(X)$ starting in $\mu(X)^+$ a Type-II path. We will elaborate on the subtlety of this path lifting from the following three perspectives, with more details given in Section \ref{sec-ex-exotic}.
\begin{itemize}
\item[(1)] ({\it Concatenation}) One can build up a path that lifts to $\mathcal L(X)$ via a series of concatenations of multiple sub-paths in either Type-I or Type-II, see the left picture in Figure \ref{figure-ex-exotic}. 
\item[(2)] ({\it Monodromy}) The path liftings are not unique. The right picture in Figure \ref{figure-ex-exotic} shows that when $X= E(a,b)$, there exist paths $\gamma$ having two lifts $\{L_t\}_{t \in [0,T]}$ and $\{L'_t\}_{t \in [0, T]}$ such that $L_T$ and $L'_T$ lie in different components of the fiber over $\gamma(T)$.
\item[(3)] ({\it Orientation}) Again, the right picture in Figure \ref{figure-ex-exotic} shows that there exist paths with $\gamma(0), \gamma(T) \in \mu(X)$ such that $\gamma$ lifts but its reserve $\overline{\gamma}: = \{\gamma(T-t)\}_{t \in [0,T]}$ does not lift. 
\end{itemize}

Next, we give both necessary and sufficient conditions (unfortunately not quite the same) for paths in the reduced (Hamiltonian) shape invariant of balls, integral ellipsoids, and polydisks to lift. For simplicity, we will state the results only for certain Type-II paths. More general paths can be considered via concatenations mentioned above (see Corollary \ref{thm-lift}).  

\begin{theorem} [Path lifting for $B^4(R)$] \label{thm-path-lift-ball} Let $\gamma = \{\gamma(t)\}_{t \in [0,T]}$ be a path in ${\rm Sh}_H^+(B^4(R))$ with $\gamma(0) \in \mu(B^4(R))^+$ but $\gamma(T) \notin \mu(B^4(R))^+$. Denote $\gamma(t) = (r_t, s_t)$, then we have the following conclusions. 
\begin{itemize}
\item[(I)] If $\frac{r_t}{s_t}$ is non-decreasing and $2r_t + s_t \geq R$ for all $t \in [0, T]$, then $\gamma$ does {\rm not} lift to ${\mathcal L}(B^4(R))$.
\item[(II)] The path $\gamma$ does lift to ${\mathcal L}(B^4(R))$ if there exists a $t_*$ with $0 \le t_* \le T$ satisfying 
\begin{itemize}
\item[(II-i)] $\gamma|_{[0, t_*]} \in \mu(B^4(R))^+$,
\item[(II-ii)] $2r_{t_*} + s_{t_*} <R$, 
\item[(II-iii)] $0< r_t < \frac{R}{3}$ for any $t \in [t_*,T]$.
\end{itemize}
\end{itemize}
\end{theorem}

\begin{theorem} [Path lifting for integral $E(a,b)$] \label{thm-path-lift-ellipsoid} Let $\gamma = \{\gamma(t)\}_{t \in [0,T]}$ be a path in ${\rm Sh}_H^+(E(a,b))$ with $k: = \frac{b}{a} \in \N_{\geq 2}$,  $\gamma(0) \in \mu(E(a,b))^+$ but $\gamma(T) \notin \mu(E(a,b))^+$. Denote $\gamma(t) = (r_t, s_t)$, then we have the following conclusions. 
\begin{itemize}
\item[(I)] If $\frac{r_t}{s_t}$ is nondecreasing and $(k+1)r_t + s_t \geq  b$ for all $t \in [0,T]$, then $\gamma$ does {\rm not} lift to ${\mathcal L}(E(a,b))$.
\item[(II)] The path $\gamma$ does lift to ${\mathcal L}(E(a,b))$ if there exists a $t_*$ with $0 \le t_* \le T$ satisfying 
\begin{itemize}
\item[(II-i)] $\gamma|_{[0, t_*]} \in \mu(E(a,b))^+$,
\item[(II-ii)] either one of the following conditions holds,
\begin{itemize}
\item[(II-ii-1)] $(k-1) r_{t_*} \leq s_{t_*}$ and $(k+1)r_{t_*} + s_{t_*} < b$;
\item[(II-ii-2)] $(k-1) r_{t_*} > s_{t_*}$ and $0<r_{t_*}<\frac{a}{2}$; 
\end{itemize}
\item[(II-iii)] $0<r_{t}<\frac{a}{2}$ for all $t \in [t_*,T]$.
\end{itemize}
\end{itemize}
\end{theorem}

\begin{theorem} [Path lifting for $P(c,d)$] \label{thm-path-lift-polydisk} Let $\gamma = \{\gamma(t)\}_{t \in [0,T]}$ be a path in ${\rm Sh}_H^+(P(c,d))$ with $\gamma(0) \in \mu(P(c,d))^+$ but $\gamma(T) \notin \mu(P(c,d))^+$. Denote $\gamma(t) = (r_t, s_t)$, then we have the following conclusions. 
\begin{itemize}
\item[(I)] If $\frac{r_t}{s_t}$ is non-decreasing and $r_t + s_t \geq d$ for all $t \in [0, T]$, then $\gamma$ does {\rm not} lift to ${\mathcal L}(P(c,d))$.
\item[(II)] The path $\gamma$ does lift to ${\mathcal L}(P(c,d))$ if there exists a $t_*$ with $0 \le t_* \le T$ satisfying 
\begin{itemize}
\item[(II-i)] $\gamma|_{[0, t_*]} \in \mu(P(c,d))^+$,
\item[(II-ii)] $r_{t_*} + s_{t_*} <d$, 
\item[(II-iii)] $0< r_t < \frac{c}{2}$ for any $t \in [t_*,T]$.
\end{itemize}
\end{itemize}
\end{theorem}

\begin{ex} \label{ex-lift} \normalfont (1) The left picture in Figure \ref{figure_ball_lift} shows a path $\gamma_1$ that does not lift to $\mathcal L(B^4(R))$, while the right one shows a path $\gamma_2$ that does lift. This is implied by Theorem \ref{thm-path-lift-ball}. 
\begin{figure}[h]
  \centering
   \includegraphics[scale=0.85]{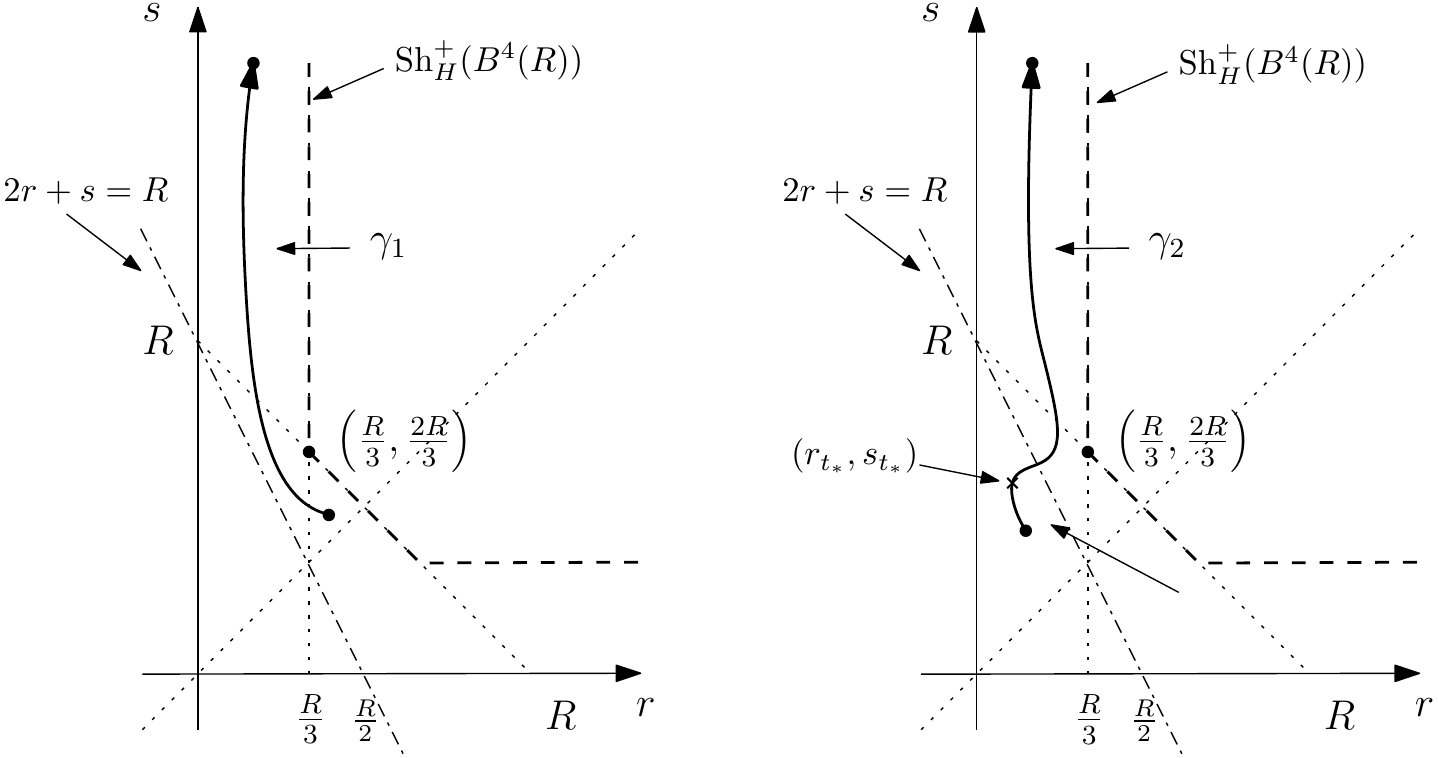} 
     \caption{Path $\gamma_1$ does not lift to $\mathcal L(B^4(R))$ but $\gamma_2$ does lift.} \label{figure_ball_lift}
\end{figure}

(2) The left picture in Figure \ref{figure_ellipsoid_lift} shows a path $\gamma_1$ that does not lift to $\mathcal L(E(a,b))$, while the right one shows a path $\gamma_2$ that does lift. This is implied by Theorem \ref{thm-path-lift-ellipsoid}. In particular, for path $\gamma_2$, the condition (II-ii-2) in Theorem \ref{thm-path-lift-ellipsoid} applies.
\begin{figure}[h]
  \centering
   \includegraphics[scale=0.9]{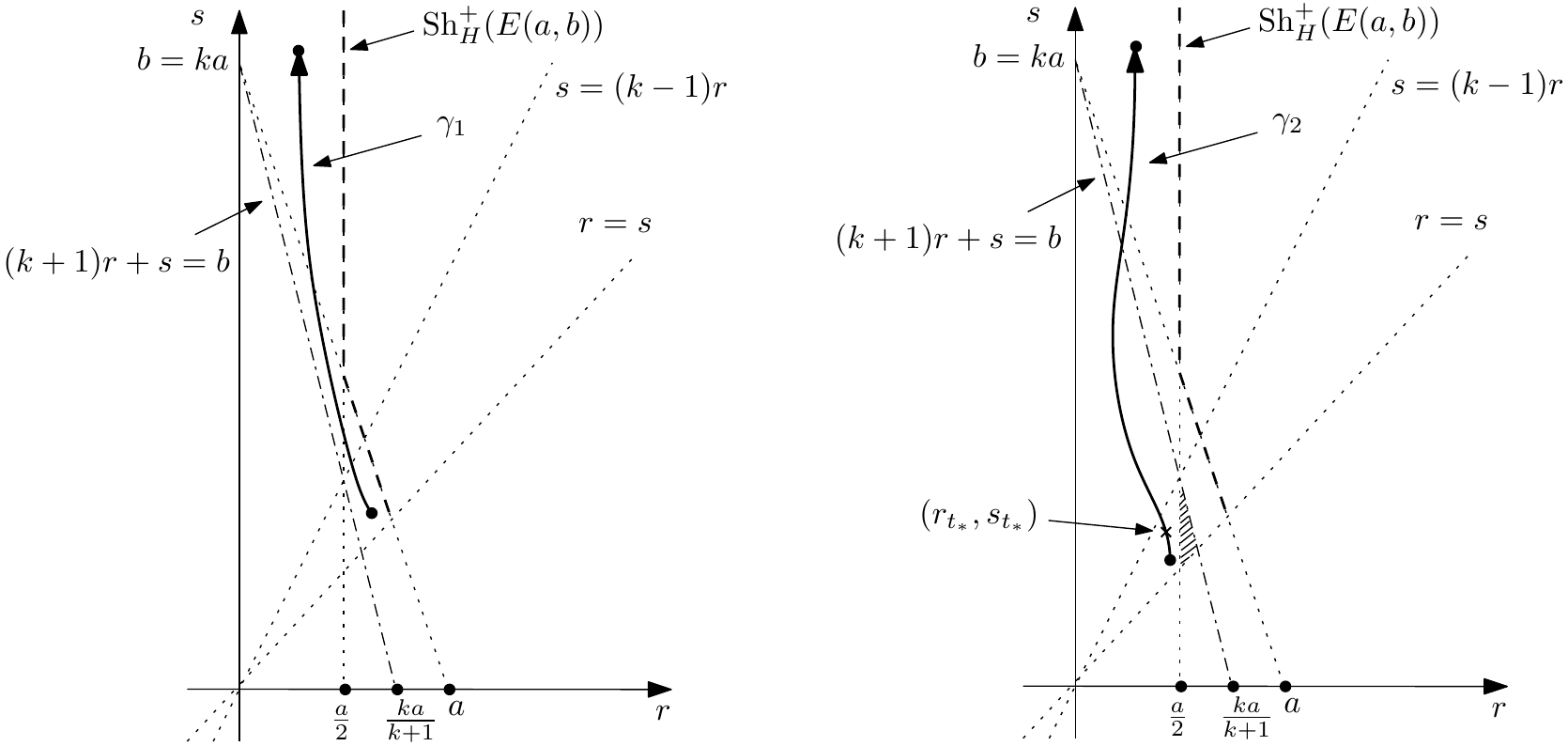} 
     \caption{Path $\gamma_1$ does not lift to $\mathcal L(E(a,b))$ but $\gamma_2$ does lift.} \label{figure_ellipsoid_lift}
\end{figure}

(3) The left picture in Figure \ref{figure_polydisk_lift} shows a path $\gamma_1$ that does not lift to $\mathcal L(P(c,d))$, while the right one shows a path $\gamma_2$ that does lift. This is implied by Theorem \ref{thm-path-lift-polydisk}. 
\begin{figure}[h]
  \centering
   \includegraphics[scale=0.9]{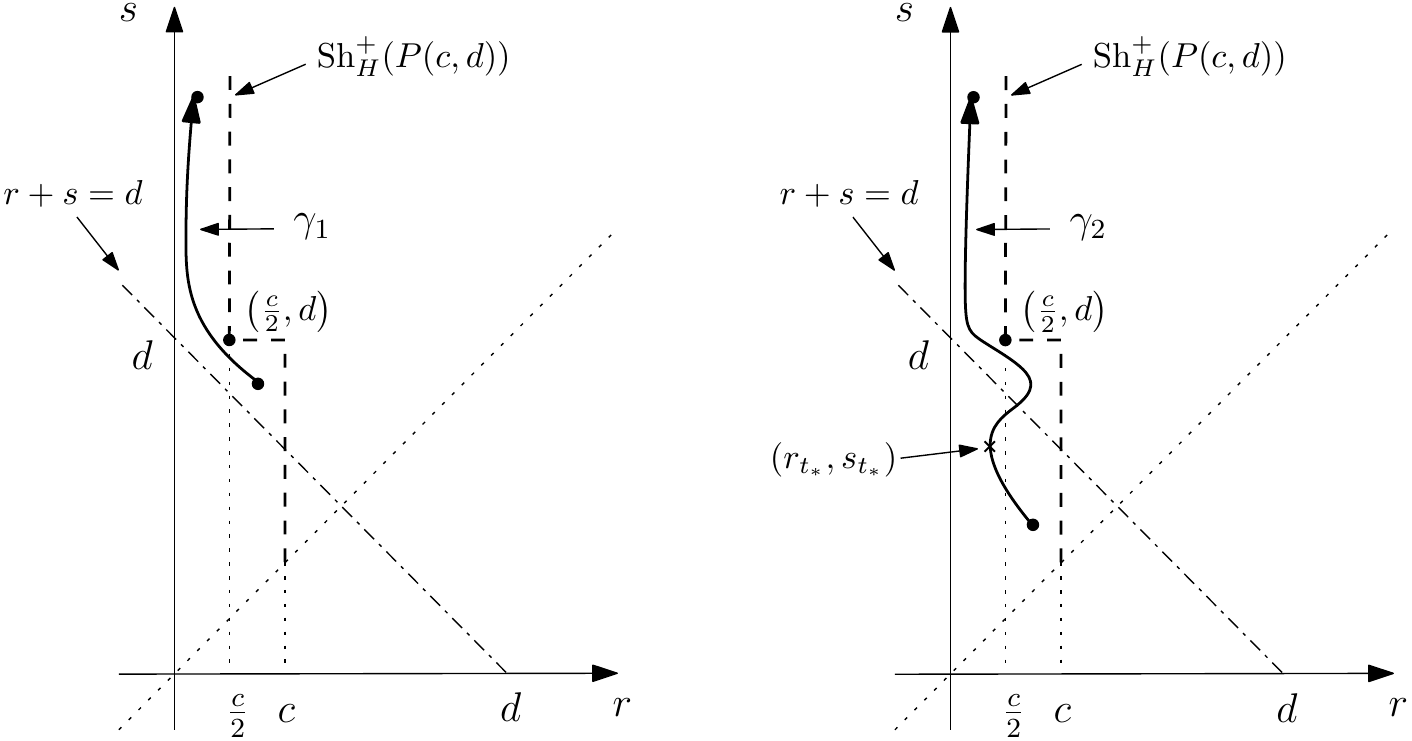} 
     \caption{Path $\gamma_1$ does not lift to $\mathcal L(P(c,d))$ but $\gamma_2$ does lift.} \label{figure_polydisk_lift}
\end{figure}

\end{ex}

\begin{remark} For $\gamma_2$ in Figure \ref{figure_ellipsoid_lift} in Example \ref{ex-lift}, if the starting point is in the shaded region, then Theorem \ref{thm-path-lift-ellipsoid} is not strong enough to determine whether this path lifts or not. \end{remark}
 
The obstructions, (I) in Theorem \ref{thm-path-lift-ball} and Theorem \ref{thm-path-lift-ellipsoid} can be viewed as a single result (where for the case $B^4(R)$, we set $k = 1$) and it will be proved in subsection \ref{ssec-path-lift-ball-ellipsoid}; (I) in Theorem \ref{thm-path-lift-polydisk} has a similar proof, given in subsection \ref{ssec-path-lift-ball-polydisk}. Results (II) in Theorem \ref{thm-path-lift-ball}, \ref{thm-path-lift-ellipsoid}, and \ref{thm-path-lift-polydisk}, which will be proved in subsection \ref{ssec-path-lift-ball-ellipsoid_2}, are consequences of a general path lifting criterion that works for any toric domains in $\R^4$, see Corollary \ref{thm-lift} in subsection \ref{ssec-path-lifting-criterion}. 

\subsection{Symplectic embeddings}\label{ssec-symp-emb}

If there exists a Hamiltonian diffeomorphism $f:X \hookrightarrow Y$ then we have ${\rm Sh}_{H}^+(X) \subset {\rm Sh}_{H}^+(Y)$ (see Proposition 7.1 in \cite{hindzhang}). Given its natural scaling properties we can therefore think of ${\rm Sh}_{H}^+(X)$ as a kind of set-valued symplectic capacity. For a possible relation between this set-valued symplectic capacity with the classical $\R_{\geq 0}$-valued symplectic capacity, see subsection 1.2.1 in \cite{hindzhang}. Some resulting obstructions to symplectic embeddings were explored in Theorem 1.6 in \cite{hindzhang}, however in the case when $X$ and $Y$ are ellipsoids the obstructions turn out to be fairly weak, and are all consequences of Gromov's non-squeezing together with the volume constraint.

Now, analyzing more closely from the path lifting perspective, we observe that if $\gamma:[0,T] \to {\rm Sh}_H^+(X)$ is a path with $\gamma(0) \in \mu(X)^+ \cap \mu(Y)^+$ then, given our symplectic embedding $\phi: X \hookrightarrow Y$, if $\gamma$ lifts as $\{L_t\}_{t \in [0,T]}$ to ${\mathcal L}(X)$ then $\{\phi(L_t)\}_{t \in [0, T]}$ gives a lift to ${\mathcal L}(Y)$, although this lift to ${\mathcal L}(Y)$ may not satisfy our usual initial condition, that is, $\phi(L_0) = L(\mathcal P(\gamma(0)))$, a product torus in $Y$. However, applying a Hamiltonian diffeomorphism of $Y$, the initial condition can be satisfied if $\phi(L_0)$ is unknotted in $Y$. Hence, we produce either examples of knotted Lagrangian tori in $Y$ or potentially stronger embedding obstructions from $X$ to $Y$. We have several consequences in these two directions. 

\subsubsection{Detecting knotted Lagrangian tori}\label{sssec-emb-knotted} The following result provides another approach (cf.~Theorem \ref{knotted-ball}, Theorem \ref{knotted-ellipsoid}, and Theorem \ref{knotted-polydisk}) to detect knotted Lagrangian tori. It will be proved in subsection \ref{obs-domain_1}. 

\begin{theorem} \label{emb-knotted} We can detect knotted Lagrangian tori in the following three cases. 
\begin{itemize}
\item[(1)] Suppose there exists a symplectic embedding $\phi:E(1,x) \hookrightarrow B^4(R)$ for $1<R<x$. If $(r,s) \in \mu(E(1,x))^+ \cap \mu(B^4(R))^+$ with $2r+s >R$, then the embedded Lagrangian torus $\phi(L(r,s))$ is knotted in the fiber $\P^{-1}((r,s))$ of $B^4(R)$. 
\item[(2)] Suppose there exists a symplectic embedding $\phi:E(1,x) \hookrightarrow E(a,b)$ for $1< a < b=ka < x$ with $k \in \N_{\ge 2}$. If $(r,s) \in \mu(E(1,x))^+ \cap \mu(E(a,b))^+$ with $(k+1)r+s > b$, then the embedded Lagrangian torus $\phi(L(r,s))$ is knotted in the fiber $\P^{-1}((r,s))$ of $E(a,b)$.
\item[(3)] Suppose there exists a symplectic embedding $\phi:E(1,x) \hookrightarrow P(c,d)$ for $1< c < d < x$. If $(r,s) \in \mu(E(1,x))^+ \cap \mu(P(c,d))^+$ with $r+s > d$, then the embedded Lagrangian torus $\phi(L(r,s))$ is knotted in the fiber $\P^{-1}((r,s))$ of $P(c,d)$.
\end{itemize}
\end{theorem}

We emphasize that the knotted Lagrangian tori produced by Theorem \ref{emb-knotted} do not overlap with the ones produced by Theorem \ref{knotted-ball} or Theorem \ref{knotted-ellipsoid} or Theorem \ref{knotted-polydisk}. Here, we provide examples to support this. 

\begin{ex} \label{emb-knotted-ex}\normalfont (1) Let $\phi: E(1,4) \hookrightarrow B^4(2)$ be a symplectic embedding. We know such an embedding exists, see \cite{MS12}. Any $(r,s)$ in the shaded region in the left picture of Figure \ref{figure_ex_knotted} satisfies the assumption in (1) of Theorem \ref{emb-knotted}. Therefore, $\phi(L(r,s))$ is knotted in the fiber $\P^{-1}((r,s))$ of $B^4(2)$. 

\smallskip

(2) Now let $k \in \N_{\ge 2}$ and consider a symplectic embedding $\phi: E(\frac{ka}{k+1}, (k+1)a) \hookrightarrow E(a,b)$. We verify in Section \ref{app} that such embeddings do exist. Then, by Theorem \ref{emb-knotted} (2) we see that $\phi(L(r,s))$ is knotted in the fiber $\P^{-1}((r,s))$ of $E(a,b)$ provided $(r,s) \in \mu(E(\frac{ka}{k+1}, (k+1)a))^+ \cap \mu(E(a,b))^+$ with $(k+1)r+s > b$, or in other words,
\begin{equation}
(r,s) \in \Delta(a,b) \cap \left\{r \le s \,\, \mbox{and} \,\, (k+1)r + \frac{ks}{k+1} < b \,\, \mbox{and} \,\,(k+1)r + s > b\right\}.
\end{equation}
Comparing to Theorem \ref{knotted-ellipsoid}, this gives additional points in the region $\frac{a}{2} < r < \frac{b}{k+1}$.

\begin{figure}[h]
  \centering
   \includegraphics[scale=0.8]{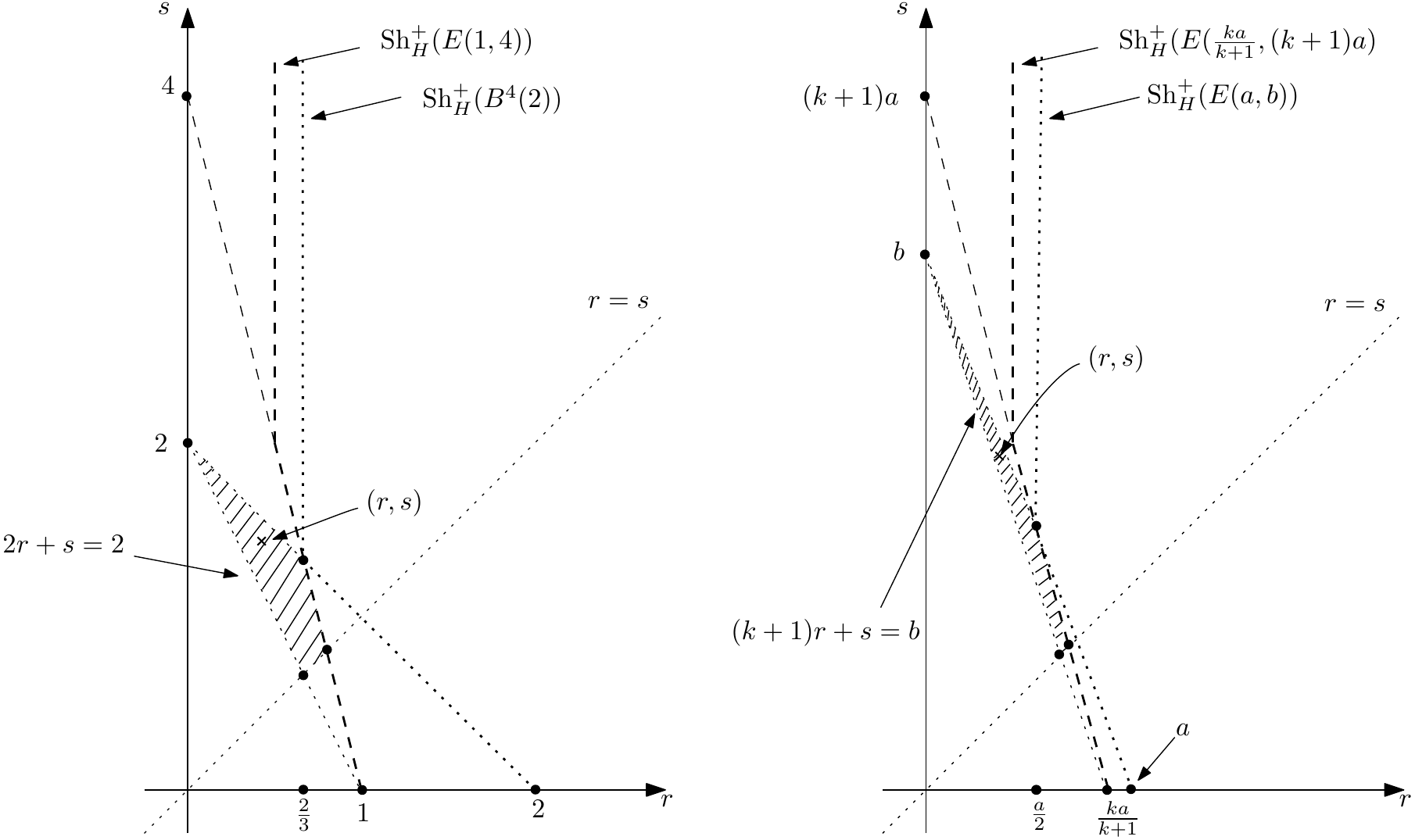} 
     \caption{Knotted Lagrangian tori in shaded regions} \label{figure_ex_knotted}
\end{figure}

\smallskip

(3) Let $\phi: E(1,4) \hookrightarrow P(1, 2)$ be a symplectic embedding. We know such an embedding exists, see \cite{C-GFS17}. Any $(r,s)$ in the shaded region in the left picture of Figure \ref{figure_ex_knotted_polydisk} satisfies the assumption in (3) of Theorem \ref{emb-knotted}. Therefore, $\phi(L(r,s))$ is knotted in the fiber $\P^{-1}((r,s))$ of $P(1, 2)$. 
\begin{figure}[h]
  \centering
   \includegraphics[scale=0.85]{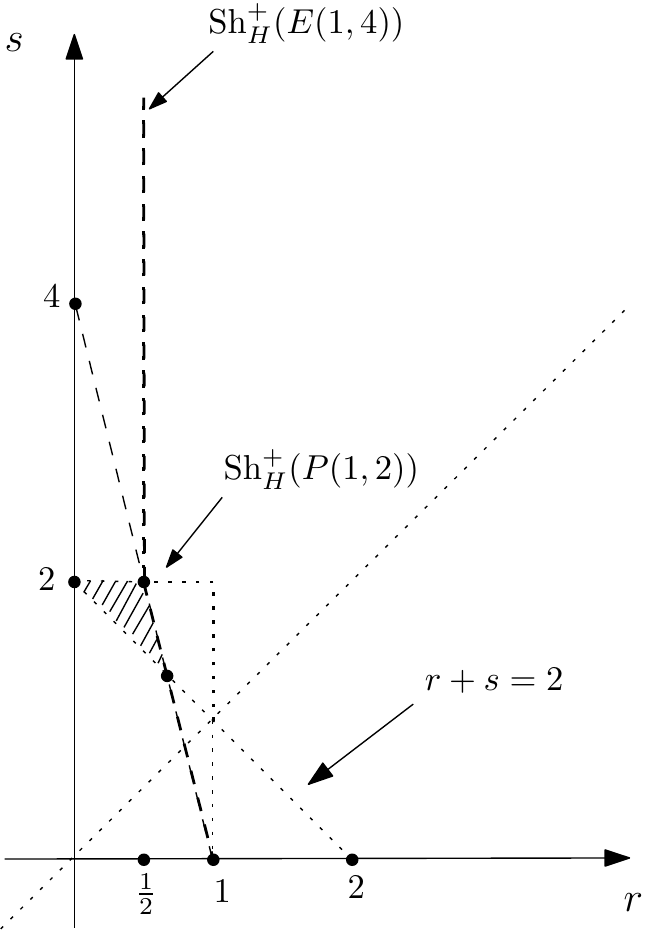} 
     \caption{Knotted Lagrangian tori in shaded regions} \label{figure_ex_knotted_polydisk}
\end{figure}

Both shaded regions in Figure \ref{figure_ex_knotted} and the shaded region in Figure \ref{figure_ex_knotted_polydisk} contain strictly more points than those given by Theorem \ref{knotted-ball}, Theorem \ref{knotted-ellipsoid}, and Theorem \ref{knotted-polydisk}. However, we need to confirm the existence of a symplectic embedding when applying Theorem \ref{emb-knotted}, which is sometimes non-trivial, see Proposition \ref{prop-ellipsoid-emb} in Section \ref{app}. \end{ex}

Here are a few immediate corollaries of Theorem \ref{emb-knotted} which can be contrasted with results on stabilized symplectic embeddings, see Theorem 1.1 in \cite{McD18} or Theorem 1.1 and Theorem 1.3 in \cite{C-GHS21}. Relations between Lagrangian isotopies and stabilized embeddings will be explored elsewhere.

\begin{cor} Denote by $\overline{E(1,x)}$ the closure of the open ellipsoid $E(a,b)$ in $\C^2$. We have the following conclusions. 
\begin{itemize}
\item[(1)] If $\overline{E(1,x)} \hookrightarrow B^4(R)$ and the image of $L(\frac{x}{x+1}, \frac{x}{x+1}) \subset \partial \overline{E(1,x)}$ is unknotted in $B^4(R)$, then $R\geq \frac{3x}{x+1}$.
\item[(2)] If $\overline{E(1,x)} \hookrightarrow E(a,b)$ with $k : = \frac{b}{a} \in \N_{\geq 2}$ and the image of $L(\frac{x}{x+k-1}, \frac{(k-1)x}{x+k-1}) \subset \partial \overline{E(1,x)}$ is unknotted in $E(a,b)$, then $a\geq\frac{2x}{x+k-1}$.
\item[(3)] If $\overline{E(1,x)} \hookrightarrow P(c,d)$ with $k: = \frac{d}{c} \in \R_{>0}$ and the image of $L(\frac{x}{x+2k-1}, \frac{(2k-1)x}{x+2k-1}) \subset \partial \overline{E(1,x)}$ is unknotted in $P(c,d)$, then $c \geq \frac{2x}{x+2k-1}$. 
\end{itemize}
\end{cor}

\begin{proof} (1) By a direct comparison of the intersections of ${\rm Sh}_H^+(E(1,x))$ and ${\rm Sh}_H^+(B^4(R))$ from Theorem \ref{shapecalc} with $\{r = s\}$, we know $\overline{E(1, x)} \hookrightarrow B^4(R)$ implies that $R \geq \frac{2x}{x+1}$. Then 
\[ \left(\frac{x}{x+1},\frac{x}{x+1}\right) \in \mu(\overline{E(1,x)})^+ \cap \mu(B^4(R))^+. \]
Hence, (1) in Theorem \ref{emb-knotted} implies that $R \geq 2 \cdot \frac{x}{x+1} + \frac{x}{x+1} = \frac{3x}{x+1}$. 

\medskip

(2) Without loss of generality, assume $1 \leq a$ and $x \geq b$. We know ${\rm Sh}_H^+(E(1,x)) \subset {\rm Sh}_H^+(E(a,b))$ due to $\overline{E(1,x)} \hookrightarrow E(a,b)$. Then the condition $L(\frac{x}{x+k-1}, \frac{(k-1)x}{x+k-1}) \subset \partial \overline{E(1,x)}$ implies that $(\frac{x}{x+k-1}, \frac{(k-1)x}{x+k-1}) \in {\rm Sh}_H^+(E(a,b))$. Therefore, 
\[ \left(\frac{x}{x+k-1}, \frac{(k-1)x}{x+k-1} \right) \in \mu^+(\overline{E(1,x)}) \cap \mu^+(E(a,b)) \]
since this point lies on the line $s = (k-1)r$ (cf.~Figure \ref{figure_ellipsoid_lift}). Hence, (2) in Theorem \ref{emb-knotted} implies that $b \geq (k-1) \cdot \frac{x}{x+k-1} + \frac{(k-1)x}{x+k-1} = \frac{2kx}{x+k-1}$. Dividing $k$ on both sides, we obtain the desired conclusion. 

\medskip

(3) Without loss of generality, assume $1\leq c \leq d \leq x$. We know that ${\rm Sh}_H^+(E(1,x)) \subset {\rm Sh}_H^+(P(c,d))$ due to $\overline{E(1,x)} \hookrightarrow P(c,d)$. Then the condition $L(\frac{x}{x+2k-1}, \frac{(2k-1)x}{x+2k-1}) \subset \partial \overline{E(1,x)}$ implies that $(\frac{x}{x+2k-1}, \frac{(2k-1)x}{x+2k-1}) \in {\rm Sh}_H^+(E(a,b))$. Therefore, 
\[ \left(\frac{x}{x+2k-1}, \frac{(2k-1)x}{x+2k-1} \right) \in \mu^+(\overline{E(1,x)}) \cap \mu^+(P(c,d)) \]
since this point lies on the line $s = (2k-1)r$. Hence, (3) in Theorem \ref{emb-knotted} implies that $d \geq \frac{x}{x+2k-1} + \frac{(2k-1)x}{x+2k-1}  = \frac{2kx}{x+2k-1}$. Dividing $k$ on both sides, we obtain the desired conclusion. \end{proof}

\subsubsection{Embedding obstructions}

The obstructions to the symplectic embedding between toric domains are usually given by certain symplectic capacities, for instance, Ekeland-Hofer capacity \cite{EH90}, ECH capacities \cite{Hut11}, Gutt-Hutchings' capacities \cite{GH18}, etc. Almost all of them are constructed via dynamical information, e.g., closed Reeb orbits, on $\partial X$ when it is viewed as a contact manifold with the contact structure induced by the standard primitive of the symplectic structure on $\R^4$.  Until now, the cases that have been studied the most are toric concave domains and toric convex domains. By using the reduced (Hamiltonian) shape invariants, we are able to obtain embedding obstructions for a large family of toric star-shaped domains that are beyond the cases of toric concave or convex (see, e.g., the toric domain from the subset in $\R_{\geq 0}^2$ bounded by the orange curve in Figure \ref{figure_obs_3}). Here is the result, which will be proved in subsection \ref{obs-domain_2}. 

\begin{theorem} \label{thm-obs-toric} Given a toric star-shaped domain $X$ in $\R^4$ and an $E(a,b)$ where $k = \frac{b}{a} \in \N_{\geq 1}$. Suppose $X \not\subset E(a,b)$. If there exists an ellipsoid $E$ satisfying the following conditions:
\begin{itemize}
\item[(i)] $E \subset X \cap E(a,b)$, and $E \not\subset E\left(\frac{ak}{k+1}, b\right)$,
\item[(ii)] there exists an oriented path 
\begin{equation} \label{exist-obt-path}
\gamma = \{(r_t, s_t) \in \R^2 \,| \, r_t \leq s_t\}_{t \in [0,T]} \subset \mu(X) \cap \mu\left(E\left(\frac{ak}{k+1}, b\right)\right)^{c}
\end{equation}
with $(r_0, s_0) \in \mu(E)$, $(r_1, s_1) \notin \mu(E(a,b))$, and the ratio $\frac{r_t}{s_t}$ non-increasing, 
\end{itemize}
then $X$ can not symplectically embed into $E(a,b)$. \end{theorem}

We illustrate the strength of Theorem \ref{thm-obs-toric} via the following corollaries. They provide obstructions to symplectic embeddings without computing any symplectic capacities. 

\begin{cor} \label{cor-emb-1} Let $E(a,b)$ be a symplectic ellipsoid with $k: = \frac{b}{a}\in \N_{\geq 2}$. If there exists a symplectic embedding $E(1,x) \hookrightarrow E(a,b)$ with $1<a < 1  +\frac{1}{k}$, then $b \geq x$. \end{cor}

\begin{proof} Suppose $x>b$, then see the left picture in Figure \ref{figure_emb_ex}. Referring to Theorem \ref{thm-obs-toric} where $X = E(1,x)$, the desired ellipsoid $E= X_{\Delta(1,b)}$ where $\Delta(1,b)$ is the blue triangle with vertices $(0,0)$, $(1,0)$ and $(0, b)$, shown in the picture, and the desired path is the red path $\gamma$ in the picture. Therefore, Theorem \ref{thm-obs-toric} implies the contradiction. \end{proof}

\begin{remark} The result in Corollary \ref{cor-emb-1} can also be derived from {\rm ECH} capacities denoted by $c_k^{\rm ECH}$, assuming the computational fact on the {\rm ECH} capacities of 4-dimensional ellipsoids,  Proposition 1.2 in \cite{Hut11}. Explicitly, suppose $b<x$, then consider the $(k+1)$-th {\rm ECH} capacity. One can verify that $c^{\rm ECH}_{k+1}(E(1,x)) > ak = b = c^{\rm ECH}_{k+1}(E(a,b)).$ \end{remark}

\begin{cor} \label{cor-emb-2} Consider toric domains $B^4(20)$ and $X_{\Omega}$ where the boundary $\partial \Omega \cap \R_{>0}^2$ is piecewise linear with vertices $(0, 24), (2, 17), (0, 19)$, see the right picture in Figure \ref{figure_emb_ex}. Then $X_{\Omega}$ can not symplectically embed into $B^4(20)$.
\begin{figure}[h]
  \centering
   \includegraphics[scale=0.8]{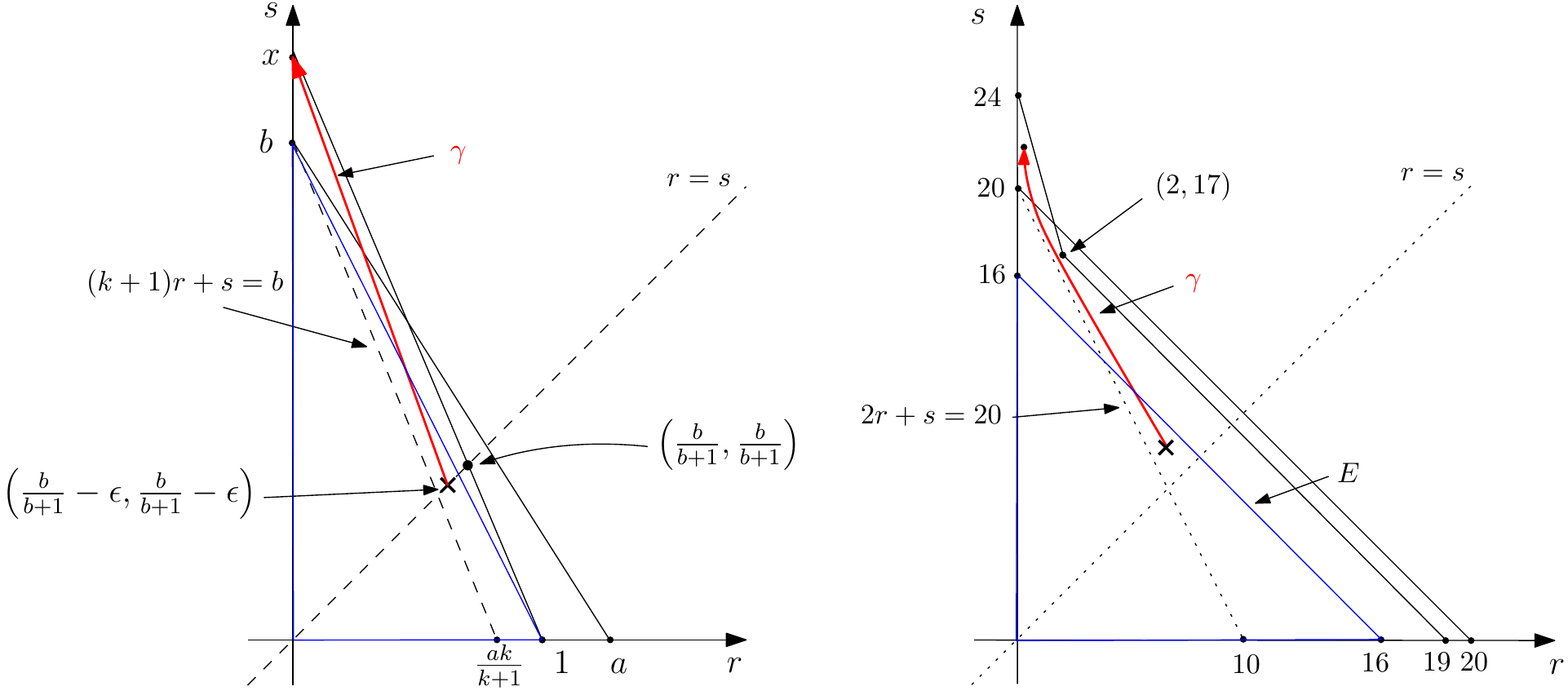} 
     \caption{Embedding obstructions from path lifting} \label{figure_emb_ex}
\end{figure}
\end{cor}

\begin{proof} By taking $E = B^4(16)$ as shown in the blue triangle in the picture and $\gamma$ as the red path in the picture, Theorem \ref{thm-obs-toric} implies the desired conclusion. \end{proof}

\begin{remark} This embedding obstruction can also be derived from Gutt-Hutchings' capacities denoted by $c_k^{\rm GH}$ and constructed in \cite{GH18}. Explicitly, consider the second Gutt-Hutchings' capacity. Since $X_{\Omega}$ is a toric concave domain, by Theorem 1.14 in \cite{GH18}, one can verify that $c_2^{\rm GH}(X_{\Omega}) = 21 > 20 = c_2^{\rm GH}(B^4(20))$. \end{remark}

\noindent{\bf Discussion.} For both Corollary \ref{cor-emb-1} and Corollary \ref{cor-emb-2}, one can deform the domain $E(1,x)$ or $X_{\Omega}$ to any star-shaped domains as long as the blue ellipsoid $E$ and the red path $\gamma$ exist, then still the embedding obstructions hold by Theorem \ref{thm-obs-toric}. However, the classical symplectic capacities may not apply at all to the deformed domains. Moreover, the following two remarks are particularly interesting. 

\smallskip

(1) In order to obtain the conclusion in (1) in Corollary \ref{cor-emb-1}, it is necessary to apply the obstruction from path lifting (instead of merely comparing the Hamiltonian shape invariants). Indeed, from ${\rm Sh}_H^+(E(1,x)) \subset {\rm Sh}_H^+(E(a,b))$, we only know that 
\[ \frac{a}{2} \geq \frac{1}{2} \,\,\,\,\mbox{and}\,\,\,\, \left(\frac{a}{2}, \frac{b}{2} \right) \,\mbox{lies above the line}\,\, rx+ s = 1.\]
In other words, $a \geq 1$ and $\frac{b}{2} \geq 1 - \frac{a}{2} \cdot x$ (which is $b \geq 2 - ax$). However,  since $x \geq 1$ and $a\geq 1$, we must have $2 - ax \leq x$. In fact, this provides void information, since $x>b \geq 2$ implies $2-ax <0$. 

\medskip

(2) The point $(2,17)$ on $\partial\Omega$ in (2) in Corollary \ref{cor-emb-2} is crucial in the sense that it lies above the line $2r+s = 20$, so there is a sufficiently large space to produce the desired path $\gamma \subset \mu\left(E(10, 20)\right)^{c}$. Curiously, the obstruction $c^{\rm GH}_2(X_{\Omega}) > c^{\rm GH}_2(B^4(20))$ shows exactly this geometric property. It would be interesting to investigate more a accurate relation between the path lifting obstruction and the capacities $c^{\rm GH}_k$.

\subsection{Related work} Both \cite{EGM18} and \cite{STV18} study the star-shape of symplectic manifolds. The star-shape is defined relative to a fixed Lagrangian torus $L_0$ and in our language describes which {\it linear paths} have lifts starting at $L_0$. One of the main results in \cite{EGM18} that relates to our work is a series of computations and estimations (via Poisson bracket invariant in \cite{BEP12}) of the star-shape with different starting Lagrangians $L_0$ in $\C^n$. With more sophisticated algebraic machinery, e.g., Fukaya algebra, \cite{STV18} enhances \cite{EGM18} in various ways. A consequence of the main result in \cite{STV18} is that for toric Fano varieties, the star-shape relative to the monotone Lagrangian fiber coincides with the moment polytope. In the case of $B(R)$ this implies that linear paths starting from $(\frac{R}{3}, \frac{R}{3}) \in \Delta(R,R)$ have lifts only if the path lies completely in the moment image $\Delta(R,R)$. We note that $(\frac{R}{3}, \frac{R}{3})$ lies precisely on the boundary of our ``flexible region'' $\{2r + s < R\}$, so the result matches Theorem \ref{thm-path-lift-ball}. In this case our results cover more general paths, and we give constructions showing the constraints are often sharp. This answers the question raised up in Example 6.1 in \cite{STV18}. Also, our work can answer the fundamental curiosity, Question 1.1 in \cite{EGM18}, at least when the ambient symplectic manifolds are certain basic toric domains in $\R^4$. Finally, we emphasize that our results rely on the non-triviality of certain moduli spaces of holomorphic curves, which are only established for simple domains. 
\medskip

\noindent{\bf Acknowledgements.} This work was completed when the second author holds the CRM-ISM Postdoctoral Research Fellow at CRM, University of Montreal, and the second author thanks this institute for its warm hospitality. The first named author is supported by Simons Foundation Grant no.~633715.

\section{Background and preliminary}

\subsection{Shape invariant} Given an exact symplectic manifold $(M^{2n}, \omega = d\lambda)$ with a fixed primitive $\lambda$, the shape invariant of this $(M, \omega = d\lambda)$, denoted by ${\rm Sh}(M, \lambda)$, is defined as the collection of all possible area classes of embedded Lagrangian tori. More explicitly, for any Lagrangian embedding $\phi: \mathbb T^n \hookrightarrow M^{2n}$, the pullback $\phi^*\lambda$ represents a cohomology class in $H^1(\T^n; \R)$. By choosing an integral basis $e: = (e_1, ...,e_n)$ of $H_1(\T^n; \Z)$, the following evaluation 
\begin{equation} \label{eva}
[\phi^*\lambda](e) = (\lambda(\phi_*(e_1)), ..., \lambda(\phi_*(e_n))) \in \R^{n} 
\end{equation}
induces a map from Lagrangian embeddings to elements in $\R^n$. Of course, the set of values from (\ref{eva}) depends on $\lambda$ and $e$, where a different choice of $\lambda$ results in a uniform shift in $\R^n$ and a different choice of $e$ results in a transformation by an element in ${\rm GL}(n, \Z)$. In particular, the action by ${\rm GL}(n, \Z)$ provides a certain symmetry of ${\rm Sh}(M, \lambda)$. It is Sikorav's work \cite{Sik89} and Eliashberg's work \cite{Eli91} that observe first the application of the shape invariant to the study of the rigidity of symplectic or contact embeddings. For further development in this direction, see \cite{Mul19,RZ20,hindzhang}. 

In this paper, we consider a restrictive version of the shape invariant, called the {\it Hamiltonian shape invariant}, and our $(M^{2n}, \omega = d\lambda) = (X, \omega_{\rm std} = d\lambda_{\rm std})$ for a star-shaped toric domain $X$ in $(\R^4, \omega_{\rm std})$. It is defined in (\ref{dfn-hsi}) above as the collection of all possible area classes $(r,s) \in \R_{>0}^2$ that admit $L(r,s) \hookrightarrow X$, and it is denoted by ${\rm Sh}_H(X)$. Observe that in this set-up, there is no dependence of the primitive (since $X$ is contractible) and we have a canonical choice of the basis $e = (e_1, e_2)$, where $e_1$ is the standard circle in $L(r,s)$, lying in the first $\R^2$-factor of $\R^4$, bounding the disk with area $r$ and $e_2$ is the standard circle in $L(r,s)$, lying in the second $\R^2$-factor of $\R^4$, bounding the disk with area $s$. Then the only symmetry we have is via the reflection $(r,s) \mapsto (s,r)$, which induces a simplified version, the {\it reduced} Hamiltonian shape invariant denoted by ${\rm Sh}_H^+(X) := {\rm Sh}_H(X) \cap \{r \leq s\}$. 

The explicit computations of ${\rm Sh}_H(X)$ as a subset in $\R_{>0}^2$, even for basic toric domains such as balls, ellipsoids and polydisks, were carried out quite recently, see \cite{HO19, hindzhang}. Theorem \ref{shapecalc} in the introduction summarizes the corresponding results. These results are consequences of sophisticated analysis on holomorphic curves via symplectic field theory (SFT). 

A slightly different version of the shape invariant, which is formulated via the flux of a Lagrangian isotopy (in particular applied on closed manifolds such as $\C P^n$ and $S^2 \times S^2$), has been studied in \cite{EGM18} and \cite{STV18}. These works point out the intriguing relations between the shape invariant and Poisson-bracket invariants in \cite{BEP12} and SYZ fibration in \cite{SYZ96}, respectively.  

\subsection{Symplectic field theory} \label{ssec-SFT} Lagrangian embeddings $\phi: L(s,t) \hookrightarrow X$ can be effectively studied via symplectic field theory (SFT). It is a modern machinery, originally formulated in the work \cite{EGH00} and further developed in \cite{BEHWZ03,Hof06,CM18}, that can associate a variety of algebraic invariants to a symplectic cobordism. Topologically, our symplectic cobordism is the complement $W: = X \backslash U_g^*L$ where $L = \phi(L(r,s))$ and $U^*_gL$ is the unit codisk bundle of $L$ with respect to some metric $g$ on $L$; our metrics will always be flat. With a preferred almost complex structure $J$ on $W$, the central objects in SFT are $J$-holomorphic curves $C: (S^2 \backslash\{p_1, ..., p_m\}, j) \to (W, J)$, where the asymptotic ends from punctures $p_i$ correspond to Reeb orbits on $\partial X$ (positive ends) and on $S_g^*L : = \partial U_g^*L$ (negative ends). In fact, the Reeb orbits on $S_g^*L$ can readily be classified (see Proposition 3.1 in \cite{HO19}). More concretely, we say a Reeb orbit on $S_g^*L$ is of the type $(-m,-n)$, denoted by $\gamma_{(m, n)}$, if its projection to $L$ is in the homology class $(-m, -n) \in H^1(\mathbb T^2, \Z)$. Note that for a $J$-holomorphic plane $C: (S^2 \backslash \{p\}, j) \to (W, J)$ with only one {\it negative} asymptotic end on $\gamma_{(-m, -n)}$, by Stoke's theorem, its area is 
\[ {\rm area}(C) = 0 - (r(-m) + s(-n)) = rm + sn (>0). \]

A useful algebraic invariant of $C$ is its Fredholm index. Denote by $\{\gamma_i^+\}_{i=1}^{s_+}$ the collection of positive asymptotic orbits and $\{\gamma_i^-\}_{i=1}^{s_-}$ the collection of negative asymptotic orbits. Without loss of generality, let us assume these Reeb orbits are either non-degenerate or Morse-Bott, that is, they may come in smooth families, and $S_i^+$ and $S_i^-$ are the leaf spaces of the associated Morse-Bott submanifolds. Fix a symplectic trivialization $\tau$ of $C^*TW$ along these Reeb orbits, and $c_1^{\tau}(C^*TW)$ denotes the first Chern number with respect to this trivialization $\tau$. Then 
\begin{align} \label{dfn-f-ind}
{\rm ind}(C) & = (s_+ + s_- - 2) + 2c_1^{\tau}(C^*TW)  \\
& \,\,\,\,\,+ \left(\sum_{i=1}^{s^+} {\rm CZ}^{\tau}(\gamma_i^+) + \frac{\dim S_i^+}{2}\right) - \left(\sum_{i=1}^{s_-} {\rm CZ}^{\tau}(\gamma_i^-) - \frac{\dim S_i^+}{2}\right), \nonumber
\end{align}
where ${\rm CZ}^{\tau}$ is the Robin-Salamon index with respect to $\tau$ (see \cite{RS93}). Note that ${\rm ind}(C)$ is independent of the choice of symplectic trivialization $\tau$. When specifying $\gamma_i^{-} = \gamma_{(-m_i, n_i)}$ for $i \in \{1, ..., s_-\}$ and taking $\tau$ as the complex trivialization of the contact planes induced by complexifying the trivialization of the 2-torus, the index formula (\ref{dfn-f-ind}) can be computed by 
\begin{equation} \label{f-ind}
{\rm ind}(C) = (s_+ + s_- -2) + \left(\sum_{i=1}^{s^+} {\rm CZ}^{\tau}(\gamma_i^+) + \frac{\dim S_i^+}{2}\right) + 2 \sum_{i=1}^{s_-} (m_i + n_i). 
\end{equation}
For more detailed elaboration, see Example 4.1 in \cite{hindzhang}. Sometimes, the toric domain $X$ can be compactified to be a closed symplectic manifold by adding certain curves at infinity. In this paper, we are particularly interested in the case of polydisks $P(c,d)$, where $P(c,d)$ can be compactified to $S^2(c) \times S^2(d)$ with two factors having areas $c$ and $d$, by adding two curves $\{\infty\} \times S^2(b)$ and $S^2(a) \times \{\infty\}$. Then we study $C$ without positive ends, but with a topological invariant given by its intersection with these two curves at infinity. We denote by $(d_1, d_2)$ the bidegree labelling the degrees of intersections. Then the corresponding formula in (\ref{f-ind}) is  
\begin{equation} \label{f-ind-2}
{\rm ind}(C) = (s_- -2) + 4 (d_1 + d_2) + 2 \sum_{i=1}^{s_-} (m_i + n_i) 
\end{equation}

A useful technique in SFT is the neck-stretching, via a sequence of deformations of almost complex structures on the symplectic cobordism $W$. The standard SFT compactness theorem in \cite{BEHWZ03} promises the existence of a limit curve $C_{\rm lim}$, more precisely, a pseudo-holomorphic building consisting of curves in different levels matched in a possibly complicated way. However, the two invariants introduced above, ${\rm area}(C)$ and ${\rm ind}(C)$, converge in the limit and behave in a rather controllable manner. To be precise, if $C_{\rm lim}$ is obtained by gluing different sub-buildings $\{C_i\}_{i =1}^n$ along matching orbits $\{\gamma_i\}_{i=1}^m$, then 
\begin{equation} \label{building-area}
{\rm area}(C_{\rm lim}) = {\rm area}(C_1) + \cdots + {\rm area}(C_n)
\end{equation}
and 
\begin{equation} \label{building-ind}
{\rm ind}(C_{\rm lim}) = \sum_{i=1}^n {\rm ind}(C_i) - \sum_{i=1}^m {\rm dim}\,S_i 
\end{equation}
where $S_i$ is the leaf space of $\gamma_i$ for $i \in \{1, ..., m\}$. For more details, see Proposition 3.3 in \cite{HO19}. In what follows, we will see that the combination of ${\rm area}(C)$ and ${\rm ind}(C)$ actually yields quite strong constraints on the possible configurations of $C_{\rm lim}$. This will be essential to the study of embeddings $L(r,s) \hookrightarrow X$. This scheme was used in both \cite{HO19} and \cite{hindzhang} for the computations of the (Hamiltonian) shape invariant.

\section{Obstructions to path liftings} 

In this section, we will prove the obstruction to a path lifting, that is, (I) in Theorem \ref{thm-path-lift-ball} and Theorem \ref{thm-path-lift-ellipsoid}. This will be divided into several steps. 

\subsection{Preparations} The following lemma, Lemma \ref{lem-1}, guarantees the existence of certain curves that initiate the proof. 

\medskip

Recall that $L(r_0, s_0)$ is the product torus in $\C^2$ with area classes $r_0$ and $s_0$. We assume $L(r_0, s_0) \subset E(a,b)$, that is $kr_0 + s_0 < b$. Recall that $\gamma_{(m,n)}$ denotes the
closed Reeb orbits of type $(-m, -n)$ on the unit co-sphere bundle $S^*_gL(r_0, s_0)$ with respect to a preferred flat metric $g$ on $L(r_0, s_0)$. Since $L(r_0, s_0) \subset E(a,b)$, up to a rescaling of the metric $g$, there exists a sufficiently small neighborhood as a unit codisk bundle $U_g^*L(r_0, s_0)$ sitting inside $E(a,b)$. As discussed in subsection \ref{ssec-SFT}, the complement $E(a,b)\backslash U_g^*L(r_0, s_0)$ is a symplectic cobordism and the deformed complex structure, denoted by $J_0$, gives a compatible almost complex structure with cylindrical ends. Denote by $\gamma_b$ the long closed Reeb orbit on $\partial E(a,b)$ with period $b$. 

\begin{lemma} \label{lem-1} There exists a $J_0$-holomorphic cylinder in the symplectic cobordism $E(a,b)\backslash U_g^*L(r_0, s_0)$ with a positive end on $\gamma_b$, the longer Reeb orbit of $\partial E(a,b)$, and with a negative end on a Reeb orbit $\gamma_{k,1}$ on $S^*_gL(r_0, s_0)$ (where recall $k = \frac{b}{a}$).  \end{lemma}

\begin{proof} There exists a thin and long ellipsoid denoted by $E(\ep, \ep S)$ that can be embedded inside the unit codisk bundle $U_g^*L(r_0, s_0)$. Here, $\ep$ is sufficiently small and $S$ is sufficiently large. Up to symplectomorphism, we have the following inclusions, 
\[ E(\ep, \ep S) \subset U_g^*L(r_0, s_0) \subset E(a,b). \]
Denote by $\beta$ the short closed Reeb orbit of $\partial E(\ep, \ep S)$ with period $\ep$. By an appropriate deformation of the almost complex structure, we can view $E(a,b) \backslash E(\ep, \ep S)$ as a symplectic cobordism with respect to an almost complex structure (still) denoted by $J_0$. By Theorem 5.8 in \cite{hindzhang}, there exists a $J_0$-holomorphic cylinder with positive end $\gamma_b$ and negative end $\beta^{k+1}$. \footnote{Strictly speaking, the statement of Theorem 5.8 in \cite{hindzhang} guarantees the existence of a curve with positive ends more than just $\gamma_b$. However, it is easy to see that the gluing method in the proof of Theorem 5.8 in \cite{hindzhang}, when applying to the cylinder provided by Lemma 5.6 in \cite{hindzhang} and a cylinder, labeled as $C_{\rm HK}$, provided by \cite{HK18-2}, can result in the desired cylinder here.} By a neck-stretching along the boundary $S_g^*L(r_0, s_0)$, the limiting building $C_{\rm lim}$ contains $(k+2)$-many curves in $E(a,b) \backslash U_g^*L(r_0, s_0)$, denoted by $\{F_1, ..., F_{k+2}\}$. Then by Lemma 6.1 in \cite{hindzhang}, for each $i \in \{1, ..., k+2\}$, we have ${\rm ind}(F_i) = 1$. Moreover, since we have only one positive end, the argument in the proof of ``only if'' part of Theorem 1.4 in \cite{hindzhang} implies that all but one of the $F_i$ are planes and the remaining curve is a cylinder with positive end on $\gamma_b$. Without loss of generality, assume $F_1$ is the cylinder. Moreover, suppose the negative end of $F_i$ has type $(-k_i, -l_i)$. Since the negative ends bound a cycle in $T^* L(r_0, s_0)$ we see
\begin{equation} \label{top-zero}
\sum_{i=1}^{k+2} k_i = \sum_{i=1}^{k+2} l_i = 0.
\end{equation}

So far, we have implicitly assumed our $J_0$ is generic in order to guarantee curves of nonnegative index. However we would like to work with a $J_0$ such that the $z_1$ and $z_2$ axes are complex (and so finite energy planes $P_1$ and $P_2$ asymptotic to $\gamma_a$ and $\gamma_b$ respectively). As these axes are disjoint from $L(r_0, s_0)$, it is easy to check that they are not covered by any of our limiting curves, and so we still have the necessary regularity.

As the axes are complex, positivity of intersection implies that for $i \ge 2$ the $F_i$ are asymptotic to a Reeb orbits representing a class with $k_i \ge 0$ and $l_i \ge 0$. As the curves have index $1$, the index formula implies that $k_i + l_i =-1$, so the only possibility is that $(k_i, l_i) = (-1,0)$ or $(k_i, l_i) = (0,-1)$. Further, by positivity of intersection, as our limiting building has intersection number $+1$ with $P_1$, at most one $F_i$ is asymptotic to an orbit in the class $(0,-1)$.

Suppose first that exactly one $F_i$ for $i \ge 2$ is asymptotic to an orbit of type $(0,-1)$. Then $F_1$ is asymptotic to an orbit of type $(k,1)$ as required.

Alternatively, all $F_i$ for $i \ge 2$ are asymptotic to orbits of type $(-1,0)$ and $F_1$ is asymptotic to an orbit of type $(k+1,0)$. In this case the limiting building has intersection number $k+1$ with $P_2$. However we can compute Siefring's generalized intersection number (see (4-3) in \cite{siefring}) to be $$P_2 \ast P_2 = k.$$ As our limiting holomorphic building has a single unmatched positive end on $\gamma_b$ and (if we compactify with a disk inside $E(\ep, \ep S)$) is homotopic relative to its compactified positive boundary to $P_2$, Theorem 2.2 from \cite{siefring} implies that $k$ is an upper bound for the intersections between the curves in our limiting building and $P_2$. This is a contradiction if the building contains $k+1$ planes asymptotic to $(-1,0)$. \end{proof}

\begin{remark} \label{rmk-lem-1} For an alternative proof of Lemma \ref{lem-1}, at least for a carefully chosen $J_0$, we observe it is in fact possible to write down such a holomorphic cylinder directly. To do this we fix a circle $\Gamma = \{\theta_1 - k \theta_2 =0 \}$ in the torus $\T^2$, and identify the fibers of the moment map $\mu$ with the same $\T^2$, suitably collapsed over the axes. In the moment image $\mu (E(a,b))$ let $\sigma$ be a curve which coincides with the line through $(r_0, s_0)$ of slope $\frac{1}{k}$ at one end, and the vertical line through $(0,b)$ at the other. Then $\sigma \times \Gamma$ is a cylinder in $E(a,b)$, which is symplectic provided $\sigma$ never has slope $-k$. We can find such paths exactly because $b > ka$. Moreover our cylinder coincides with the trivial cylinder over the Reeb orbits at either end, so we can find an almost complex structure $J_0$ making the cylinder holomorphic.
\end{remark}

Given an isotopy of Lagrangian submanifolds $\{L_t\}_{t \in [0,1]}$, corresponding compatible almost complex structures $J_t$ on the complement of $L_t$, and a sequence of $J_{t_n}$-holomorphic curves $\{C_n\}_{n \in \N}$ with (negative) ends on the Lagrangian submanifolds $\{L_{t_n}\}_{n \in \N}$, the standard SFT-compactness implies that if $t_n \to t_*$ and the $C_n$ have bounded area, for example if they appear in the same moduli space, then the 
$C_n$ converge to a $J_{t_*}$-holomorphic building $C_{\rm lim}$. To be precise, the standard SFT-compactness applies to a sequence of ${J_n}$-holomorphic curves $C_n$ with a fixed end, however, Fukaya's trick (see subsection 2.1 in \cite{STV18}) transfers the moving boundary conditions on $L_{t_n}$ to a sequence of almost complex structures on a fixed cobordism. Hence, we still obtain a compactness result due to \cite{BEHWZ03}. A careful study of $C_{\rm lim}$ plays an important role in the proof of (I) in Theorem \ref{thm-path-lift-ball} and Theorem \ref{thm-path-lift-ellipsoid}. In particular, the following technical lemma, Lemma \ref{homology_class}, is useful to us. 

Let $C$ denote a curve with ${\rm ind}(C) = e$, one positive end on $\gamma_b$, and $e$-many negative ends on a Lagrangian torus. For instance, the cylinder provided by Lemma \ref{lem-1} satisfies this condition, since by (\ref{f-ind}) 
\begin{align*}
{\rm ind}(C) = (2k+3)  - (2k+2) = 1 = \# \,\mbox{negative ends of $C$}.
\end{align*}
Denoted by $\mathcal M_{C}(t)$ the moduli space corresponding to this curve $C$ with moving boundary conditions on $L_t$ for $t \in [0,1]$. Suppose $\mathcal M_C(t)$ is not compact, that is, there exist a degeneration at $t_* \in [0,1]$ with a limit curve $C_{\rm lim} \notin \mathcal M_C(t_*)$. By Lemma 3.7 in \cite{HO19}, there exists a curve in $E(a,b) \backslash U_g^*L_{t_*}$, denoted by $C_0$, with ${\rm ind}(C_0) \geq \# \,\mbox{negative ends of $C_0$}$. Also, we denote by $C_1, ..., C_k$ the components (as sub-buildings consisting of curves in $E(a,b) \backslash U_g^*L_{t_*}$ and in the symplectization $T^*L_{t_*}\backslash 0_{L_{t_*}}$) of the complement of $C_0$ in $C_{\rm lim}$ such that each $C_i$ matches with $C_0$ at only one negative end of $C_0$. See Figure \ref{figure_c_limit} for an example of $C_{\rm lim}$. Moreover, assume each $C_i$ for $i \in \{1, ..., k\}$ admits $e_i$-many negative ends, while the cardinality of the unmatched ends of $C_0$ is denoted by $e_0$ (hence, the total number of the negative ends of $C_0$ is $\sum_{i=0}^k e_i$).  
\begin{figure}[h]
  \centering
      \includegraphics[scale=0.75]{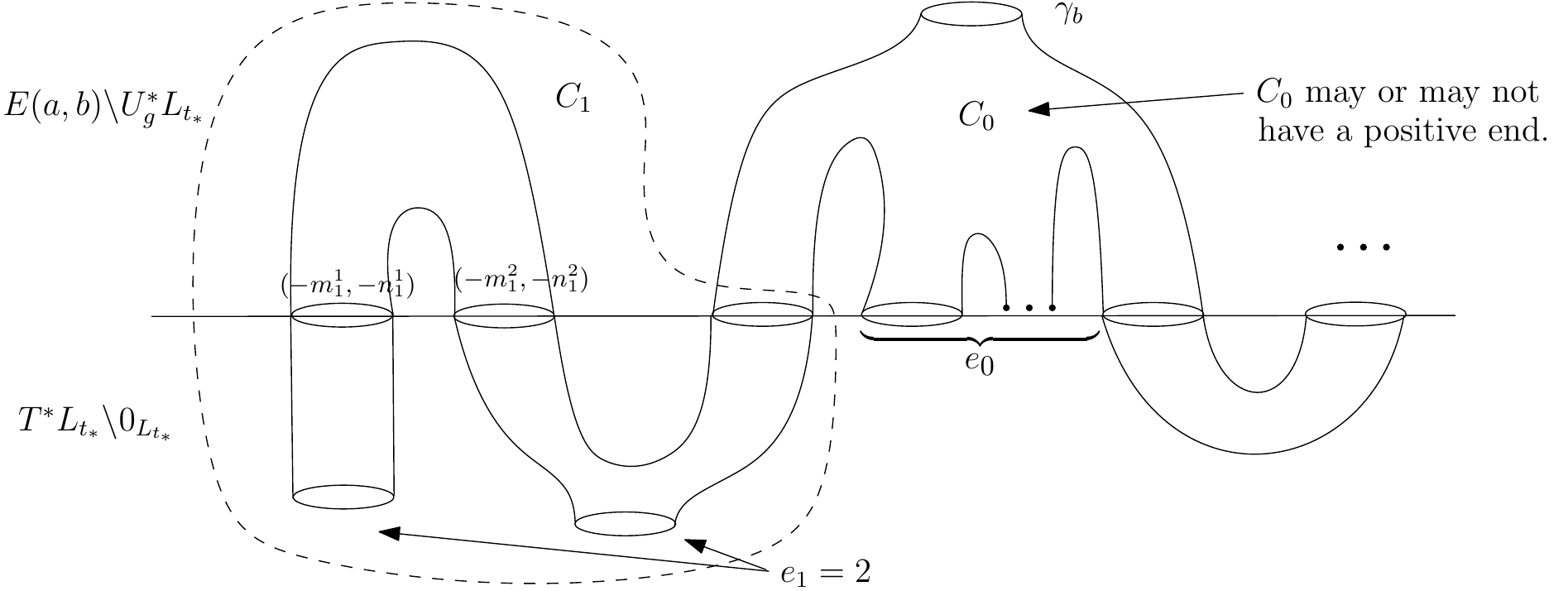}
  \caption{A limit curve in $\mathcal M_C(t_*)$.} \label{figure_c_limit}
\end{figure}
Finally, as shown in Figure \ref{figure_c_limit}, denote by the negative ends of the curves in $E(a,b) \backslash U_g^*L_{t_*}$ from the component $C_i$ by $\{(-m_i^{j}, -n_i^{j})\}_{j = 1, ..., l_i}$. Then we have the following result. 

\begin{lemma} \label{homology_class} Suppose $C_0$ contains the positive end $\gamma_b$. Then, with $(-m_i^j, -n_i^j)$ defined as above, we have $\sum_{i=1}^k \sum_{j=1}^{l_i} (m_i^j + n_i^j) \leq 0$.  \end{lemma}

\begin{proof} First, by the index matching formula (\ref{building-ind}), we have 
\begin{equation} \label{matching-index}
e = {\rm ind}(C_{\rm lim}) = \left(\sum_{i=0}^k {\rm ind}(C_i) \right) - k, \,\,\,\,\mbox{which is} \,\,\,\, \sum_{i=0}^k {\rm ind}(C_i) = e+k. 
\end{equation}
Since $e = \#\,\mbox{negative ends of $C_{\rm lim}$} = e_0 + \cdots + e_k$, by regrouping these terms, (\ref{matching-index}) is equivalent to the relation $\left({\rm ind}(C_0) - (e_0+k)\right) + \sum_{i=1}^k ({\rm ind}(C_i) - e_i) =0.$ Moreover, since $e_0 +k$ equals the $\# \, \mbox{negative ends of $C_0$}$, by assumption, ${\rm ind}(C_0) \geq e_0+k$, so 
\begin{equation} \label{ind-c_0}
\sum_{i=1}^k ({\rm ind}(C_i) - e_i) \leq 0. 
\end{equation}
Second, let us focus on a component $C_i$ for $i \in \{1, ..., k\}$. Since $C_0$ already occupies the only positive end $\gamma_b$, each such $C_i$ only has ends on $L_{t_*}$. By a further decomposition, suppose $C_i$ consists of $Q_i$-many curves in the symplectization $T^*L_{t_*} \backslash 0_{L_{t_*}}$, denoted by $\{u_i^j\}_{j=1, ..., Q_i}$. Each $u_i^j$ has $e_i^j$-many negative ends and $s_i^j$-many positive ends. In particular, $\sum_{j=1}^{Q_i} e_i^j = e_i$. Similarly, $C_i$ consists of $R_i$-many curves in $E(a,b) \backslash U_g^*L_{t_*}$, denoted by $\{v_i^j\}_{j=1, ..., R_i}$. Each $v_i^j$ has $r^j_i$-many negative ends on $L_{t_*}$, denoted by $\{(-m_i^{j,l}, -n_i^{j, l})\}_{l = 1, ..., r_i^j}$. See Figure \ref{figure_c_i} for an example of this decomposition of $C_i$. Note that $\sum_{j=1}^{Q_i} s_i^j = (\sum_{j=1}^{R_i} r_i^j) + 1$ since there is one extra end matching with $C_0$. 
\begin{figure}[h]
  \centering
      \includegraphics[scale=0.7]{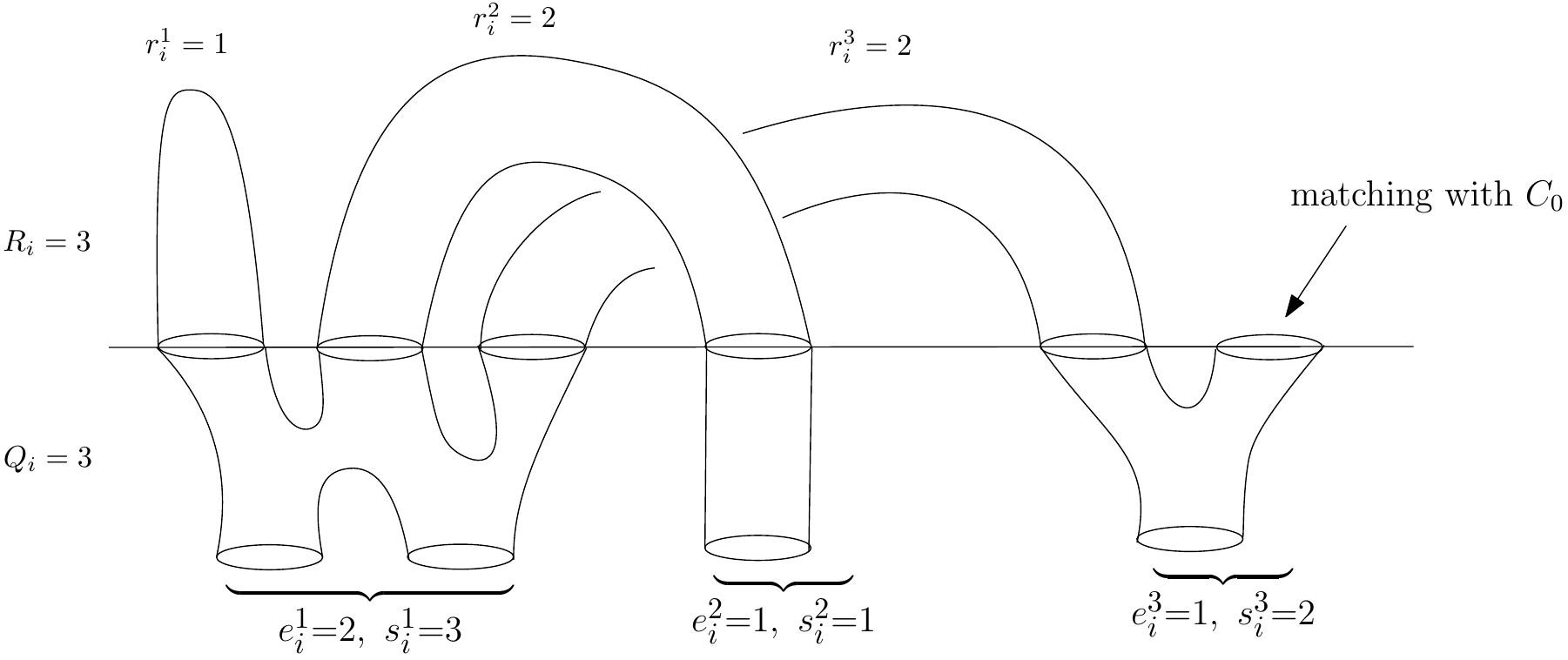}
  \caption{A decomposition of $C_i$.} \label{figure_c_i}
\end{figure}
By b) in Proposition 3.4 in \cite{HO19}, any $u_i^j$ as a curve in the symplectization has its index ${\rm ind}(u_i^j) = 2s_i^j + e_i^j -2$. Also, any $v_i^j$ as a curve in $E(a,b) \backslash U_g^*L_{t_*}$ has its index ${\rm ind}(v_i^j) = r_i^j -2 + 2 \sum_{l=1}^{r_i^j} (m_i^{j,l} + n_i^{j, l})$. Then by the index matching formula (\ref{building-ind}) again,  
\begin{align*}
{\rm ind}(C_i) - e_i & = \left(\sum_{j=1}^{R_i} {\rm ind}(v_i^j) + \sum_{j=1}^{Q_i} {\rm ind}(u_i^j) - \sum_{j=1}^{R_i} r_i^j \right) - e_i\\
& = \sum_{j=1}^{R_i} \left(r_i^j -2 + 2 \sum_{l=1}^{r_i^j} (m_i^{j,l} + n_i^{j, l})\right) \\
&\,\,\,\,\,\,\, + \sum_{j=1}^{Q_i} \left(2s_i^j + e_i^j -2 \right) - \sum_{j=1}^{R_i} r_i^j  - \sum_{j=1}^{Q_i} e_i^j\\
& = 2 \left(- R_i + \sum_{j=1}^{R_i} \sum_{l=1}^{r_i^j}(m_i^{j,l} + n_i^{j, l}) - Q_i + \sum_{j=1}^{Q_i} s_i^j \right).
\end{align*}
On the other hand, for such a decomposition of $C_i$, we can associate a tree to it by adding a vertex for each curve and asymptotic end, and an edge between the vertex for each curve and its asymptotic end. Here, we require that two vertices coincide if the corresponding asymptotic ends are matched. Since the Euler characteristic of a tree is always $1$, we have 
\[ R_i + Q_i - \left(\sum_{j=1}^{Q_i} s_i^j -1\right) = 1, \,\,\,\,\mbox{which implies} \,\,\,\, Q_i + R_i - \sum_{j=1}^{Q_i} s_i^j  = 0. \]
Therefore, we obtain the following relation, 
\[ {\rm ind}(C_i) - e_i = 2 \left(\sum_{j=1}^{R_i} \sum_{l=1}^{r_i^j}(m_i^{j,l} + n_i^{j, l}) \right).\]
Finally, sum over all $i = \{1, ..., k\}$, we have 
\begin{align*}
2 \sum_{i=1}^k \sum_{j=1}^{l_i} (m_i^j + n_i^j) & = 2 \sum_{i=1}^k \sum_{j=1}^{R_i} \sum_{l=1}^{r_i^j}(m_i^{j,l} + n_i^{j, l}) \\
& = 2 \sum_{i=1}^k ({\rm ind}(C_i) - e_i) \leq 0, 
\end{align*}
where $l_i = \sum_{j=1}^{R_i} r_i^j$ and the last step comes from (\ref{ind-c_0}). Thus we complete the proof.\end{proof}

An immediate corollary of Lemma \ref{homology_class} is the following useful result. Denote by $\mathcal M_{C_0}(t)$ the moduli space of the curve $C_0$ in Lemma \ref{homology_class} (in particular, it admits the positive end $\gamma_b$) with moving boundary conditions on $L_t$. Recall that $L_t$ is a Lagrangian isotopy covering a path $\gamma(t)$ in the shape of $X$. Denote by ${\rm area}_{C_0}(t)$ the time-dependent area of curve $C_0$ for $t \in [0,1]$, similarly to ${\rm area}_C(t)$ where $C$ is the initial curve we start from. We note that these area formulas only depend on the moduli space of $C$ or $C_0$, actually only on the homology class of $C_0$ and $C$. Thus they are well defined even if the corresponding moduli space is empty, and in particular ${\rm area}_C(t)$ is defined even for $t > t_*$. Recall that $t_* \in [0,1]$ is the moment where $C$ degenerates. Moreover, since $\gamma|_{t=1} \notin \mu(E(a,b))$ but $\gamma|_{t=1} \in {\rm Sh}_H^+(E(a,b))$, due to the picture of ${\rm Sh}_H^+(E(a,b))$, the path $\gamma$ has to pass through a point (possibly, several points) on the line segment 
\begin{equation}\label{L} 
{\rm L} : = \partial \mu(E(a,b)) \cap \left\{(r,s) \in \R_{> 0}^2 \,\big|\, 0< r < \frac{a}{2} \right\}. 
\end{equation}

\begin{prop} \label{prop-area} Suppose, at the degeneration moment $t_*$, the area class $(r_{t_*}, s_{t_*}) \in {\rm int}(\mu(E(a,b))^{+})$ and $r(t) / s(t)$ is non-increasing with respect the parameter $t \in [0,1]$. Then ${\rm area}_{C_0}(t) \leq {\rm area}_{C}(t)$ for all $t \in [t_*,1]$. 
\end{prop}

\begin{proof} We know that ${\rm area}_{C_0}(t_*) \leq {\rm area}_{C_{\rm lim}}(t_*) = {\rm area}_{C}(t_*)$ since $C_0$ is a component of $C_{\rm lim}$, it suffices to prove the conclusion for $t \in (t_*,1]$. Suppose not, that is, there exists some $t' \in (t_*, 1]$ 
such that ${\rm area}_{C_0}(t') > {\rm area}_{C}(t')$. Note that for any $t \in [t_*,1]$, we have 
\begin{align*}
{\rm area}_C(t) & = \sum_{i=1}^k {\rm area}_{C_i}(t) \\
& = {\rm area}_{C_0}(t) + \left(\sum_{i=1}^k\sum_{j=1}^{l_i} m_i^j\right) r_t + \left(\sum_{i=1}^k\sum_{j=1}^{l_i} n_i^j\right) s_t,
\end{align*}
where, by the homological reason, the curves in the symplectization contribute zero to the area. For brevity, let $M := \sum_{i=1}^k\sum_{j=1}^{l_i} m_i^j$ and $N: = \sum_{i=1}^k\sum_{j=1}^{l_i} n_i^j$.

Arguing by contradiction, the relation ${\rm area}_{C_0}(t') > {\rm area}_{C}(t')$ above implies that $Mr_{t'} + Ns_{t'} <0$. Meanwhile, since at $t= t_*$ each $C_i$ for $i \in \{1, ..., k\}$ is a holomorphic curve, we have $\sum_{i=1}^k {\rm area}_{C_i}(t_*) >0$, which is equivalent to $Mr_{t_*} + Ns_{t_*} >0$. 
Hence, since the function $Mr_t + Ns_t$ is continuous with respect to time $t$, the mean value theorem implies that there exists $t_{\dagger} \in (t_*, t')$ such that 
\begin{equation} \label{area-contra}
Mr_{t_{\dagger}} + N s_{t_{\dagger}} = 0.
\end{equation}
There are two cases as follows, see Figure \ref{figure_rel_pos}, depending on the relative position between $(r_{t_*}, s_{t_*})$ and the line $Mr + Ns  =0$. 
\begin{figure}[h]
  \centering
      \includegraphics[scale=0.8]{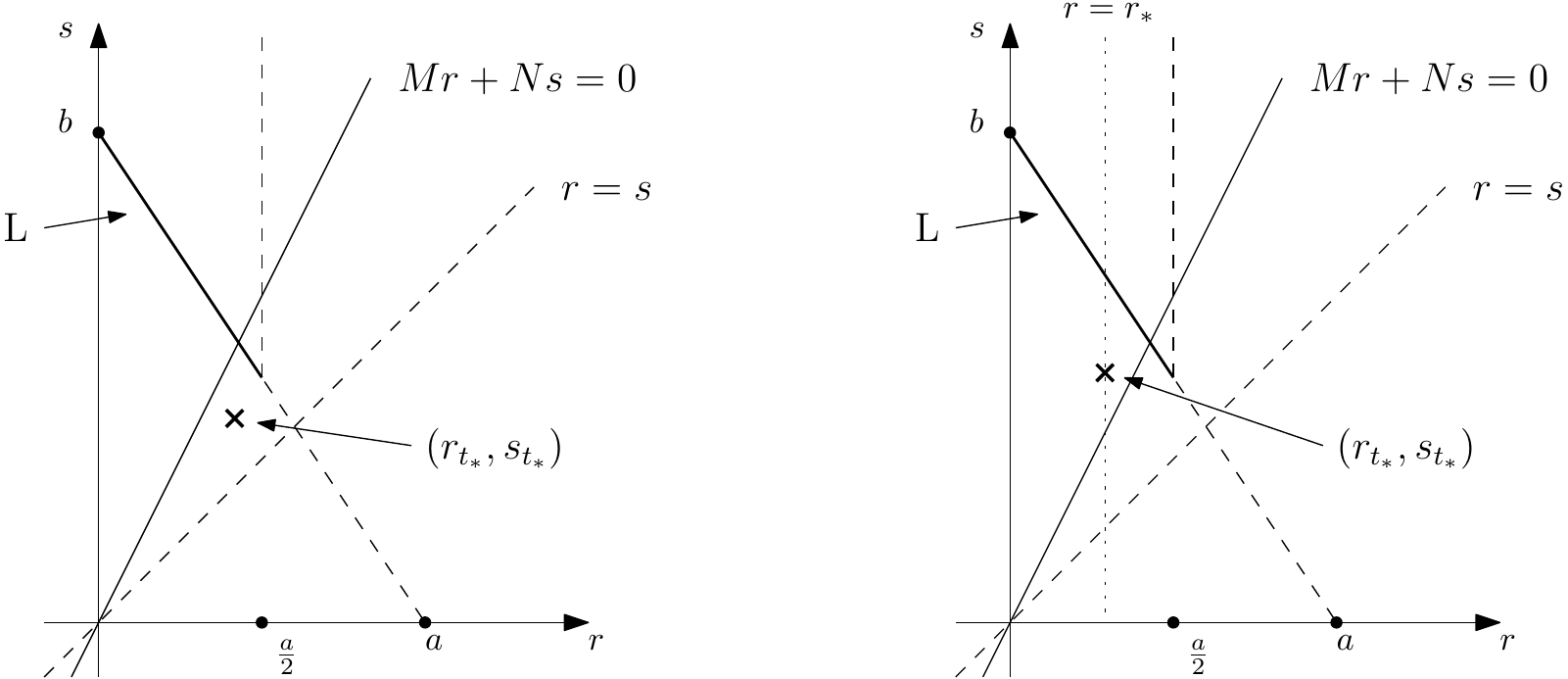}
  \caption{Relative position between $(r_{t_*}, s_{t_*})$ and $Mr + Ns = 0$.} \label{figure_rel_pos}
\end{figure}
In the left picture, the condition $Mr_{t_*} + N s_{t_*} >0$ implies that the {\it lower} half plane given by the division of the line $Mr + Ns = 0$ is the positive region. In particular, for any $r = s(>0)$, we have 
\[ Mr + Nr >0 \,\,\,\,\,\mbox{which implies that $M + N >0$}. \]
This contradicts Lemma \ref{homology_class}. In the right picture, the condition $Mr_{t_*} + N s_{t_*} >0$ implies that the {\it upper} half plane given by the division of the line $Mr + Ns = 0$ is the positive region. By our hypothesis that 
$\frac{r(t)}{s(t)}$ is non-increasing, $\gamma|_{t >t_*}$ remains above the line $Mr + Ns = 0$, contradicting the existence of a $t_{\dagger}$. 
\end{proof}


\subsection{Proof of (I) in Theorem \ref{thm-path-lift-ball} and \ref{thm-path-lift-ellipsoid}} \label{ssec-path-lift-ball-ellipsoid} Suppose by contradiction that path $\gamma = \{(r_t, s_t)\}_{t \in [0,1]}$ admits a lift $\{L_t\}_{t \in [0,1]}$. 


Denote by ${\mathcal M}_1(t)$ the moduli space consisting of cylinders with positive end $\gamma_b$ and negative end $\gamma_{k,1}$ on the moving boundary $L_t$. By Lemma \ref{lem-1}, we know that ${\mathcal M}_1(0) \neq \emptyset$. Meanwhile, by (\ref{f-ind}) each $C^{(1)}_0 \in \mathcal M_1(0)$ has its index equal to 
\begin{align*}
{\rm ind}(C^{(1)}_0) & = (1+1-2) + 0 + (2 + 2k + 1) - \left(2(k+1) + \frac{1}{2} - \frac{1}{2}\right) \\
& = (2k+3)  - (2k+2) = 1.
\end{align*}
Meanwhile, since ${\rm ind}(C^{(1)}_0) = 1> -1 = 2 \cdot {\rm genus} (C^{(1)}_0) - 2 + \# \Gamma_{\rm even}$, where $\# \Gamma_{\rm even}$ is the number of asymptotic ends with even CZ-indices, it verifies the automatic regularity of any $C^{(1)}_0 \in \mathcal M_1(0)$ (see Theorem 1 in \cite{Wen10}). Then ${\mathcal M}_1(t)$ is nonempty for an open subset of $t \in [0,1]$. 

Now, suppose that the moduli space is compact. Then $\mathcal M_1(0) \neq \emptyset$ implies that ${\mathcal M}_1(1) \neq \emptyset$. However, any curve $C^{(1)}_1 \in \mathcal M_1(1)$ has
\begin{equation} \label{area-contr}
{\rm area}(C^{(1)}_1) = b - kr_1 - s_1 >0, \,\,\,\mbox{which means} \,\,\,\frac{r_1}{a} + \frac{s_1}{b} < 1.
\end{equation}
This is in contradiction to our assumption that $(r_1, s_1) \notin \mu(E(a,b))^+$. Hence, the moduli space is not compact and we have a degeneration at some time $t_* \in [0,1]$, that is, we have a holomorphic building $C^{(1)}_{\rm lim}$ which is a limit of curves $C_n \in \mathcal M_1(t_n)$ with $t_n \to t_*$. 
The argument (\ref{area-contr}) implies that the degeneration point $(r_{t_*}, s_{t_*})$ in fact occurs inside ${\rm int}(\mu(E(a,b)^+)$. 

By Section 3.3 in \cite{HO19}, $C^{(1)}_{\rm lim}$ contains a top level curve, denoted by $C_0^{(2)}$, of index at least $e$ and with $e$ negative ends. We claim that $C_0^{(2)}$ must have a positive end; we can then check that for index reasons it must be asymptotic to $\gamma_b$. In fact, suppose not, then by Lemma 3.7 in \cite{HO19} we may assume the curve $C_0^{(2)}$ has area at least $r_{t_*}$. As any curve in ${\mathcal M}_1(t_*)$ has area $b - kr_{t_*} - s_{t_*}$, we must then have $r_{t_*} \leq {\rm area}_{C^{(2)}_0}(t_*) < b - kr_{t_*} - s_{t_*}$, which is equivalent to $(k+1)r_{t_*} + s_{t_*} < b$, contradicting our assumption. 

Denote by $\mathcal M_2(t)$ the moduli space corresponding to the curve $C^{(2)}_0$ above with the moving boundary condition on $L_t$ for $t \in [t_*,1]$. If needed, consider $\mathcal M_2(t)$ with extra marked points, so the resulting index is exactly equal to $e$ (instead of at least $e$). Then Proposition \ref{prop-area} implies that $C^{(2)}_0$ will not survive until $t =1$. In fact, it must degenerate before time $t_1$, where $\gamma|_{t=t_1}$ intersects ${\rm L}$ for the first time. Indeed, otherwise there will exists some curve denoted by $C^{(2)}_1 \in \mathcal M_2(t_1)$ that results in the following contradiction,
\[ {\rm area}_{C_1^{(2)}}(t_1)>0 \,\,\,\,\mbox{but} \,\,\,\, {\rm area}_{C_1^{(2)}}(t_1) \leq {\rm area}_{C^{(1)}_{\rm lim}}(t_1) = 0. \]
Therefore, the moduli space $\mathcal M_2(t)$ is not compact over $[t_*,1]$ and it will degenerate again at some $t_{**} \in (t_*, t_1)$. Repeat the argument above by considering the limit building $C^{(2)}_{\rm lim}$ from a degeneration of $C^{(2)}_0$ and pick the preferred top curve as above denoted by $C^{(3)}_0$, we will get a further degeneration, now within the time interval $(t_{**}, t_1)$, by Proposition \ref{prop-area}.

We note that this process must terminate, and hence the desired contradiction is given by an inductive argument. To see this, suppose we have a sequence of $C^{(k)}_0$ asymptotic to $L_k \to L_{\infty}$. Using Fukaya's trick again there exists a family of global diffeomorphisms mapping our $L_k$ to an $\tilde{L}_{\infty}$, where $\tilde{L}_{\infty}$ lies very close to $L_{\infty}$ but has rational area class, that is, $\tilde{L}_{\infty}$ is Hamiltonian isotopic to an $L(\delta m, \delta n)$ with $m,n \in \N$. We may assume that the almost complex structures on the complement of the $L_k$ push forward to tame almost complex structures on the complement of  $\tilde{L}_{\infty}$.
By construction, our $C^{(k+1)}_0$ have a positive end, and in a degeneration of $C^{(k)}_0$ the curve  $C^{(k+1)}_0$ is the only curve in the limiting building with an end on $\partial E(a,b)$. Therefore, 
the other top level curves in the limit, when pushed forward to the complement of $\tilde{L}_{\infty}$, have area at least $\delta$. Hence, in the complement of $\tilde{L}_{\infty}$, the areas of the $C^{(k)}_0$ decrease by at least $\delta$ at each step, and so indeed our recursion is finite. \qed

\subsection{Proof of (I) in Theorem \ref{thm-path-lift-polydisk}} \label{ssec-path-lift-ball-polydisk} The proof of the obstruction to the path lifting in the polydisks $P(c,d)$ is quite similar to the proof in balls and integral ellipsoids. The only variation is that we start from a curve that is different from the one produced by Lemma \ref{lem-1} or Remark \ref{rmk-lem-1}. Here we give the sketch of the proof. 

\medskip

Suppose by contradiction that path $\gamma = \{(r_t, s_t)\}_{t \in [0,1]}$ admits a lift $\{L_t\}_{t \in [0,1]}$. Compactify $P(c,d)$ to $S^2(c) \times S^2(d)$ with two factors having areas $c$ and $d$. Denote by $\mathcal M_1(t)$ the moduli space consisting of finite energy planes that intersects the $\{\infty\} \times S^2(d)$ only once (i.e., it is of bidegree $(0,1))$ and has its negative asymptotic end $\gamma_{(0,1)}$ on the moving boundary $L_t$ for $t \in [0,1]$. The moduli space $\mathcal M_1(0)$ is non-empty since we can simply write out (the image of) a curve $C$ explicitly as follows, 
\[ C^{(1)}_0 = \{(z,w) \in \C^2 \,| \, z = {\rm constant}, \,\, \pi|w|^2 >d\}, \]
and by (\ref{f-ind-2}) ${\rm ind}(C^{(1)}_0) = -1 + 2\cdot 2 +2\cdot (-1) = 1$, plus the automatic regularity can be verified. 

Now, suppose that the moduli space is compact, then $\mathcal M_1(1) \neq \emptyset$. Then any curve $C^{(1)}_1 \in \mathcal M_1(1)$ satisfies ${\rm area}(C_1^{(1)}) = d - s_1 >0$, which contradicts the our assumption that $(r_1, s_1) \notin \mu(P(c,d))^+$. Hence, we have a degeneration at some time $t_* \in [0,1]$ with a limit curve as a holomorphic building $C_{\rm lim}^{(1)}$. By Section 3.3 in \cite{HO19}, $C_{\rm lim}^{(1)}$ contains a top level curve  $C_0^{(2)}$ of index at least $e$ with $e$ negative ends. This curve $C_0^{(2)}$ intersects $\{\infty\} \times S^2(d)$ since otherwise 
\[ r_{t_*} \leq {\rm area}_{C_0^{(2)}}(t_*) < d - s_{t_*} \]
which contradicts our assumption. 

Denote by $\mathcal M_2(t)$ the moduli space corresponding to the curve $C_0^{(2)}$ above with moving boundary on $L_t$ for $t \in [t_*,1]$. It is readily checked that the same conclusion in Lemma \ref{homology_class} holds if we replace the assumption ``contains the positive end $\gamma_b$'' with ``intersects $\{\infty\} \times S^2(d)$''. Then Proposition \ref{prop-area}, whose argument only involves areas, implies that $\mathcal M_2(t)$ is not compact, so it will degenerate again at $t_{**} \in (t_*, t_1)$ where $t_1$ is the first time that $\gamma$ intersects the bar $\{(r, d) \in \R_{>0}^2 \,| \, 0<r< \frac{c}{2}\}$. The rest of the proof goes exactly the same as the one in subsection \ref{ssec-path-lift-ball-ellipsoid}, and thus we get the desired conclusion. \qed

\section{Construction of path lifting} 

\subsection{Path lifting criterion} \label{ssec-path-lifting-criterion} In this section we construct our path lifts. 
The main result is the following theorem giving a Hamiltonian isotopy with controlled support between an inclusion and a {\it rolled up} Lagrangian. The domains of interest are the $Q(a,b)$ defined here.

\begin{dfn} The domain $Q(a,b) \subset \C^2$ is defined by $Q(a,b) = \Omega_{q(a,b)} := \{(z,w) \in \C^2 \,|\, (\pi|z|^2, \pi|w|^2) \in q(a,b)\}$ where 
\begin{equation} \label{quadr}
q(a,b) = \left\{(x,y) \in \R^2 \, \bigg| \,  0 \leq x \leq 2a - \frac{ay}{a+b}, \,\, 0 \leq y \leq a+b \right\},
\end{equation}
a quadrilateral region in $\R^2$.
\end{dfn}

\begin{theorem} \label{construction} Let $0 < a <b$ and $\ep>0$. Then there exists a Hamiltonian isotopy of the Lagrangian tori
denoted by $\{L_t\}_{t \in [0,1]}$ with $L_0 = L(a,b)$ such that the following conclusions hold. 
\begin{itemize}
\item[(1)] For every $t \in [0,1]$, the Lagrangian $L_t \subset Q(a + \ep,b + \ep)$.
\item[(2)] $L_1 \subset Q(a,a)$. 
\end{itemize}
\end{theorem}

\begin{proof} We use coordinates $(z,w)$ on $\C^2$. Denote by $D(2a+2\ep)$ the disk in $z$-plane centered at origin and with area equal to $2a + 2\ep$ where $0 < \ep < \frac{b-a}{4}$. 

Consider a time-independent Hamiltonian function on $z$-plane denoted by $G(z)$ with its support in $D(2a+\ep)$ such that $0 \leq G(z) \leq a+\frac{\ep}{2}$ and its Hamiltonian diffeomorphism $\phi_G^1$ displaces $D(a)$. Note that this $G(z)$ exists since the displacement energy of $D(a)$ is $a$ and such a displacement can be obtained inside $D(2a+2\ep)$. 

Next we define a product Lagrangian torus $\tilde{L}(a,b)$ in $\C^2$ by
\[ \tilde{L}(a,b) :=\partial D(a) \times \partial ([0,1] \times [0,b]),\]
where $[0,1] \times [0,b]$ is a rectangle in $w$-plane. Let $w = x + \sqrt{-1} y$.
Of course, $\tilde{L}(a,b)$ is symplectomorphic to our standard product by a symplectomorphism on $\C$ that interchanges the disk $D(b)$ with the rectangle $[0,1] \times [0,b]$ in the $w$-plane.

Consider a product Hamiltonian function 
\[ H(z,w) = \chi(y) \cdot G(z), \]
where $\chi: [0,a+\ep] \to [0,1]$ is a smooth function, extended by zero outside $[0, a+\ep]$, such that for a sufficiently small $\delta>0$ such that $8 \delta < \ep$, 
\begin{itemize}
\item[(i)] $\chi|_{[0,\delta]} = 1$, $\chi|_{[a+\ep-\delta, a+\ep]} = 0$, and $\chi|_{[2\delta, a+\ep-2\delta]}$ is linear.
\item[(ii)] $\chi|_{[\delta,2\delta]}$ is concavely decreasing and $\chi|_{[a+\ep-2\delta, a+\ep-\delta]}$ is convexly decreasing. 
\end{itemize}
See Figure \ref{figure_chi} for the pictures of $\chi(y)$ and $-\chi'(y)$. 
\begin{figure}[h]
  \centering
      \includegraphics[scale=0.85]{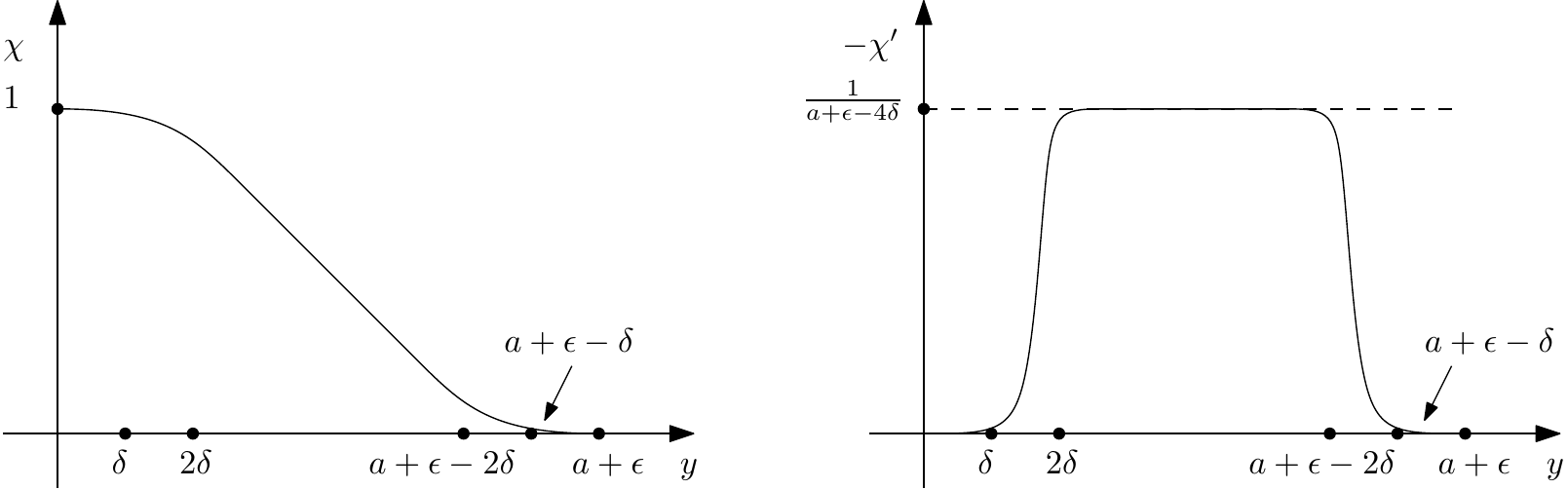}
  \caption{$\chi(y)$ and $-\chi'(y)$.} \label{figure_chi}
\end{figure}
In particular, we have the differential $dH(z,w) = \chi'(y) G(z) dy + \chi(y) dG(z)$. Under the standard symplectic structure on $\C^2$, that is, $\omega_{\rm std} = \omega_{\rm std, \C}+dx \wedge dy$, one can solve the Hamiltonian vector field $X_H$ (via the Hamiltonian equation $-dH = \iota_{X_H} \omega_{\rm std}$), which gives 
\[ X_H(z,w) = - \chi'(y) G(z) \frac{\partial}{\partial x} + \chi(y) X_{G(z)}(z), \]
where $X_{G(z)}$ is the Hamiltonian vector field generated by the function $G(z)$ above with respect to $\omega_{\rm std, \C}$. Therefore, the resulting Hamiltonian diffeomorphism of $H$ is 
\begin{equation} \label{ham-flow-H}
\phi^t_H(z,w) = (\phi_{\chi \cdot G}^t(z), (x- t\chi'(y)G(z), y)),
\end{equation}
where recall that $\phi_G^1$ is the Hamiltonian diffeomorphism of $G$ on the $z$-plane. 

Now, apply the Hamiltonian diffeomorphism $\phi^1_H$ to $\tilde{L}(a,b)$, and Figure \ref{figure_proj_z} shows a schematic picture of how the projection to the  
$w$-plane changes (indicated by the shaded part). For later use, we label the rectangle $S: = [0,1] \times [a+\ep, b]$. By our choice of $\ep$ and $\delta$ as above,
\[ \max_{L(a,b)} \{-\chi'(y) G(z)\} \leq \frac{a + \frac{\ep}{2}}{a+\ep - 4\delta} <1.\]
Therefore the bottom part of the rectangle which is moved by $\phi_H^1$ is contained in the rectangle  $[0,2] \times [0, a+\ep]$ and touches neither the line $x = 1$ nor $x=2$. 
\begin{figure}[h]
  \centering
      \includegraphics[scale=0.8]{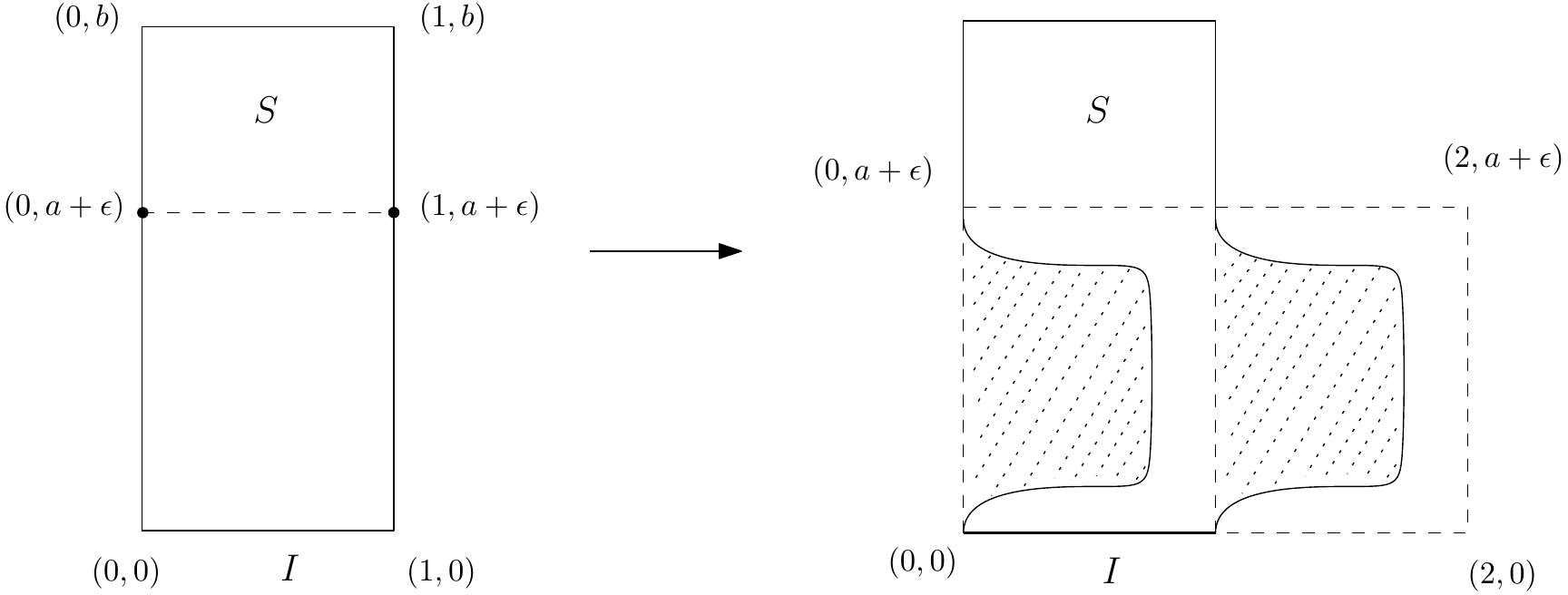}
  \caption{Behavior of $\phi_H^1(L(a,b))$ on the $w$-plane.} \label{figure_proj_z}
\end{figure}
\medskip

\noindent On the other hand, for the behavior on the $z$-plane, there are two extreme cases. 
\begin{itemize}
\item[(a)] For $w \in I$, i.e., $w = (x,0)$ for $x \in [0,1]$, since $\chi(0) = 1$, the first factor in (\ref{ham-flow-H}) is simply $\phi_G^1(z)$ which by definition displaces $\partial D(a)$ (since it displaces $D(a)$). 
\item[(b)] For $w$ near $S \cap \partial([0,1] \times [a+\ep,b])$ where $S$ is the shaded region, since $\chi(y) = 0$, the second factor in (\ref{ham-flow-H}) is identity, so $\partial D(a)$ stays the same. 
\end{itemize}

Next, consider a Hamiltonian isotopy $\{\psi_t\}_{t \in [0,1]}$ on the $w$-plane that wraps the region $S$ around the deformed $w$-projection to overlap the rectangle $[0,1] \times [0,a+\ep]$, see Figure \ref{figure_wrap}.
\begin{figure}[h]
  \centering
      \includegraphics[scale=0.8]{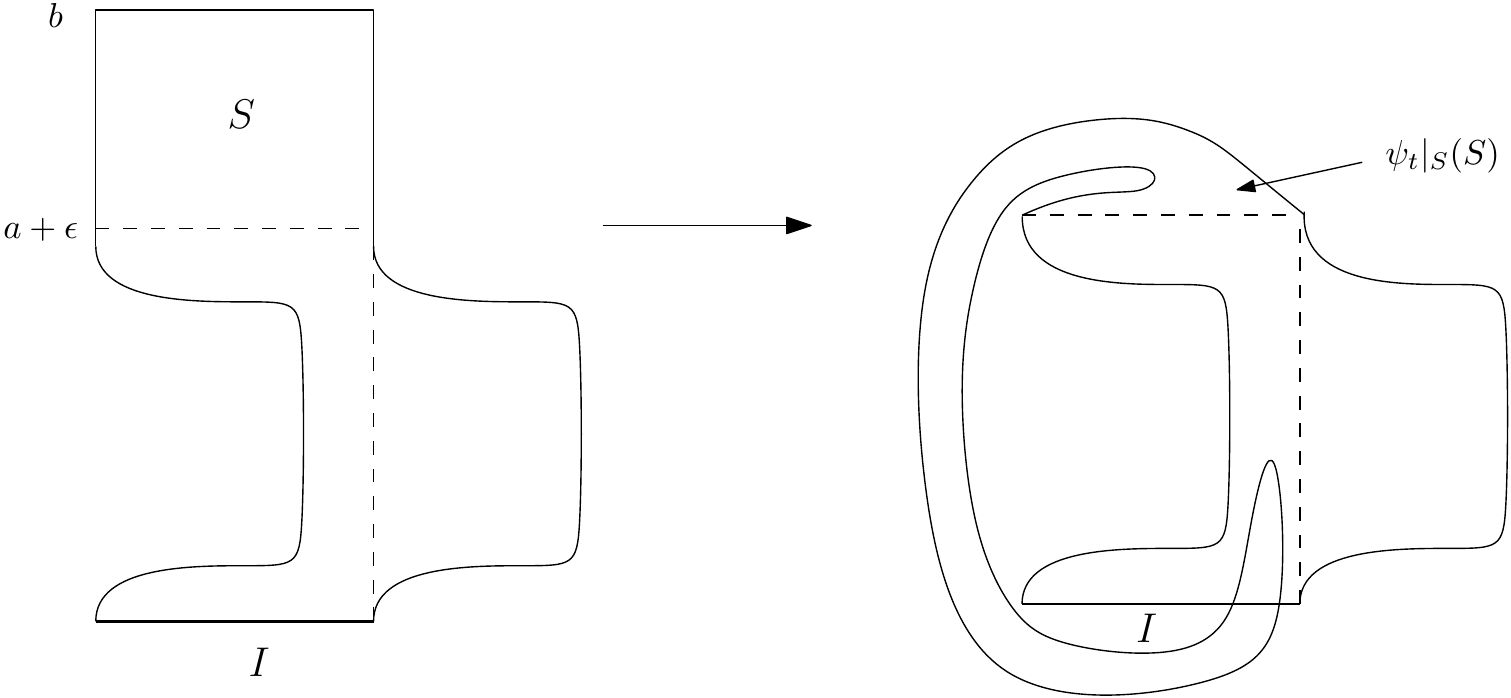}
  \caption{Wrap the region $S$.} \label{figure_wrap}
\end{figure}
Although on the $w$-plane, there are intersection points of $\psi_t|_{S}(S)$ with the line segment $I$, by our construction, viewed in the 4-dimension there will be no extra intersections. Indeed, for any point $(z,w)$ with $w$ near $S \cap \partial([0,1] \times [0,a+\ep])$, the corresponding $z \in D(a)$, but by (a) above for any point $(z,w)$ with $w \in I$, the corresponding $z$ lies in $\phi_G^1(D(a))$ which is disjoint from $D(a)$. Inductively applying this wrapping construction as shown in Figure \ref{figure_more_wrap}, then, up to an arbitrarily small neighborhood of the sides $\{0\} \times [0,1]$, $[0, a+ \ep] \times \{1\}$ and $[0,1] \times \{a+\ep\} $, we can wrap the entire $S \cap \partial([0,1] \times [a+\ep, b])$ into the shaded region inside $[0,1] \times [0,a+\ep]$ as shown in Figure \ref{figure_more_wrap}.
\begin{figure}[h]
  \centering
      \includegraphics[scale=0.7]{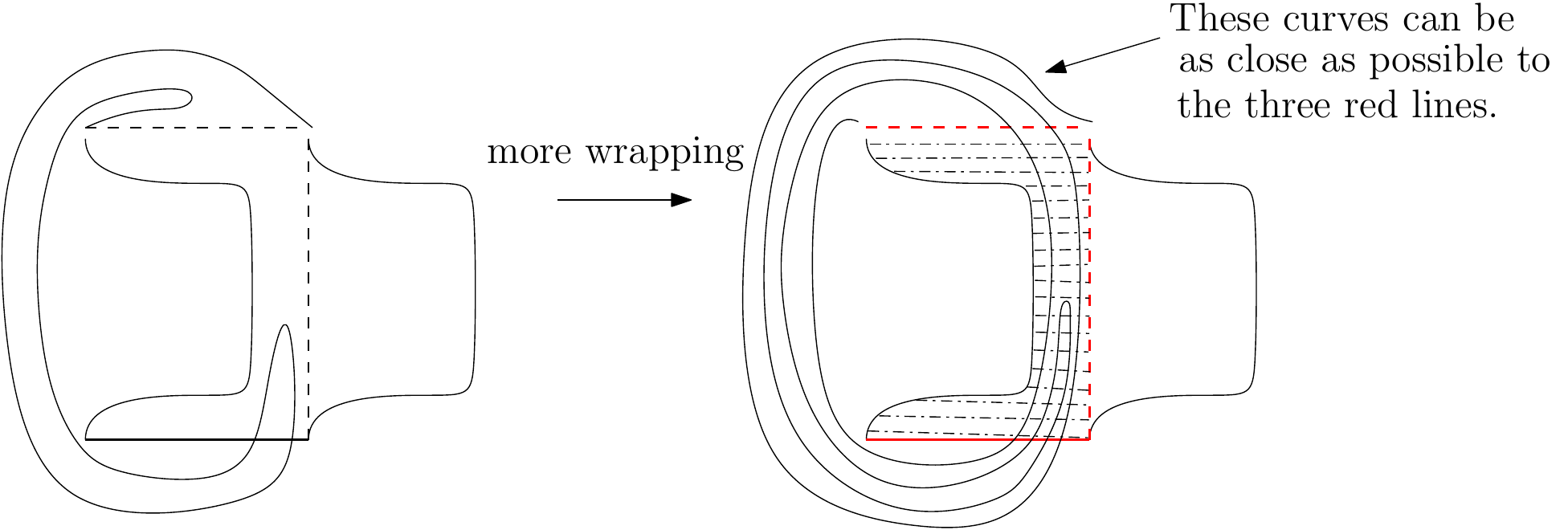}
  \caption{More wrappings.} \label{figure_more_wrap}
\end{figure}
Here, we emphasize that the extra room from $a+\ep$ to, say, $a+ 2 \ep$, that is used to construct this wrapping is guaranteed by the hypothesis that $b>a$ (and our choice $\ep< \frac{b-a}{4}$, so $a+ 2 \ep<b$). 

The Hamiltonian isotopy of interest is the image, under a symplectomorphism $\mathds{1} \times \Phi$, of the concatenation of the two constructions above,
\begin{equation} \label{isotopy}
\{\phi_H^t(\tilde{L}(a,b))\}_{t \in [0,1]} \#\{\psi_t(\phi_H^1(\tilde{L}(a,b)))\}_{t \in [0,1]}.
\end{equation}
Here $\Phi$ is a symplectomorphism of the $w$-plane that maps 

(i) $[0,2] \times [0,a+\eps] \cup [0,1] \times [a,y]$ into $D(a+y+\eps)$ for $y \geq a+\ep$, and 

(ii) each the horizontal rectangle $[0,2] \times [0,y]$ into $D(2y+2\ep)$ for $y \in [0,a+\ep]$. 
See Figure \ref{figure_rec_disk} for a layer-by-layer picture of such a map $\Phi$ for (ii). The existence of such symplectomorphisms follows Lemma 3.1.5 in the book \cite{Sch-book}.
\begin{figure}[h]
  \centering
      \includegraphics[scale=0.8]{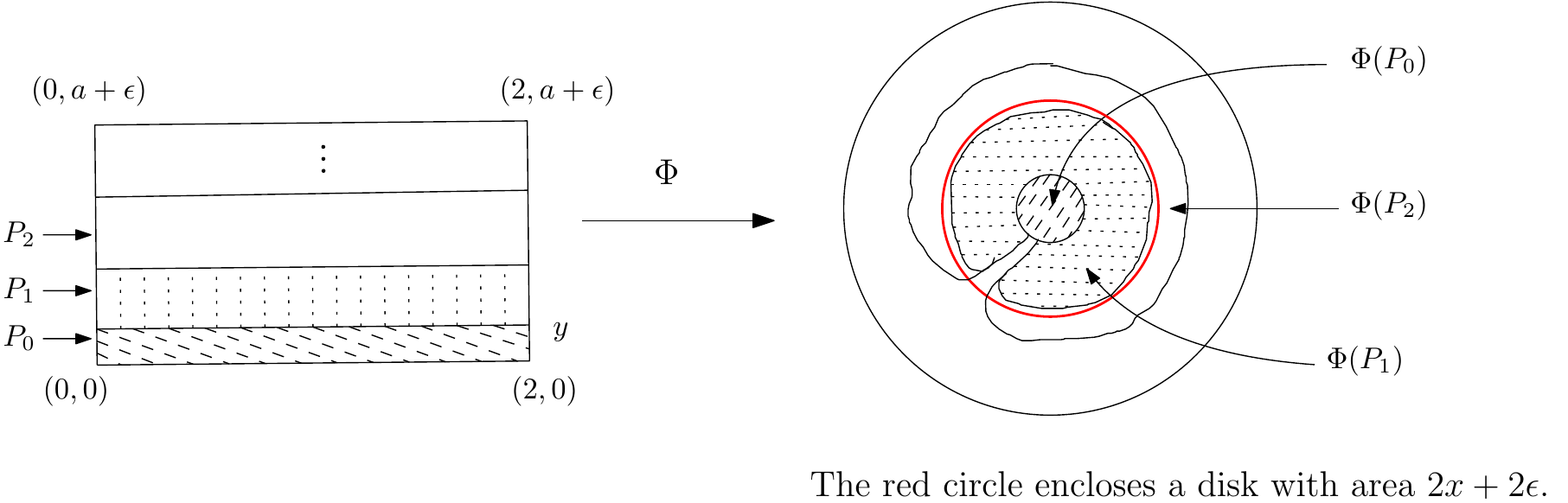}
  \caption{The map $\Phi$ on $[0,2] \times [0, a+\ep]$.} \label{figure_rec_disk}
\end{figure}

Note that both $\Phi( [0,1] \times [0,b] )$ and $D(b)$ lie inside $D(a+b+\ep)$. Therefore without loss of generality, we may assume $\Phi( [0,1] \times [0,b] ) = D(b)$, that is, our isotopy starts at the standard product $L(a,b)$.

We need to keep track of the moment image $(\pi|z|^2, \pi|w|^2)$ along the Hamiltonian isotopy and have the following observations. 
\begin{itemize}
\item[Case $\alpha$.] Let $(z,w) = (z, (x,y)) \in \tilde{L}(a,b)$ with $y \in [0, a+ \ep]$. We denote the image of $(z,w)$ under $\phi^t_H$ by $(z', w')$, and this point is described by (\ref{ham-flow-H}). In particular the $y$ coordinate is invariant under the flow. Hence, under the symplectomorphism $\mathds{1} \times \Phi$, from (\ref{ham-flow-H}) we get the relation, 
\begin{equation} \label{disk-rel}
\pi|\Phi(w')|^2 \leq 2y + \ep \,\,\,\,\mbox{and}\,\,\,\, z' \in D(2a - y + C(y,a) \ep)
\end{equation}
for some constant $C(y,a)$ only depending on $y$ and $a$. This constant $C(y,a)$ can be calculated based on the values of function $\chi(y)$. In particular, when $y = a + \ep$, we have $C(y,a) = 1$ (which means that $\pi|z|^2 \in D(a)$ as expected). The relation (\ref{disk-rel}) yields
\begin{align*}
2a - y + C(y,a) \ep & \leq 2a - \frac{\pi |\Phi(w')|^2 - \ep}{2} + C(y,a) \ep  \\
& = 2a - \frac{\pi |\Phi(w')|^2}{2} + \left( C(y,a) + \frac{1}{2} \right) \ep,
\end{align*}
and then, importantly, 
\begin{align*} 
\pi|z'|^2 & \leq 2a - y + C(y,a) \ep \\
& \leq 2a - \frac{\pi |\Phi(w')|^2}{2} + \left( C(y,a) + \frac{1}{2} \right) \ep.
\end{align*}
Such points $(z', w')$ are fixed by the flow $\psi_t$.

\item[Case $\beta$.] Let $(z,w)  = (z, (x,y)) \in \phi_H^1(\tilde{L}(a,b))$ with $y \in [a +2\ep, b]$. Such points are fixed by $\phi^H_y$, and the flow of $\psi_t$ fixes the $z$ coordinate and leaves the $w$ coordinate inside an $\eps$ neighborhood of $[0,1] \times [0,b]$. 

\end{itemize}

Therefore, we obtain a Lagrangian isotopy 
with its moment image lying in regions corresponding to either Case $\alpha$ or Case $\beta$ as above. The following graph illustrates these regions, for brevity, when $\ep \to 0$. 
\begin{equation*} \label{moment-image}
 \includegraphics[scale=0.8]{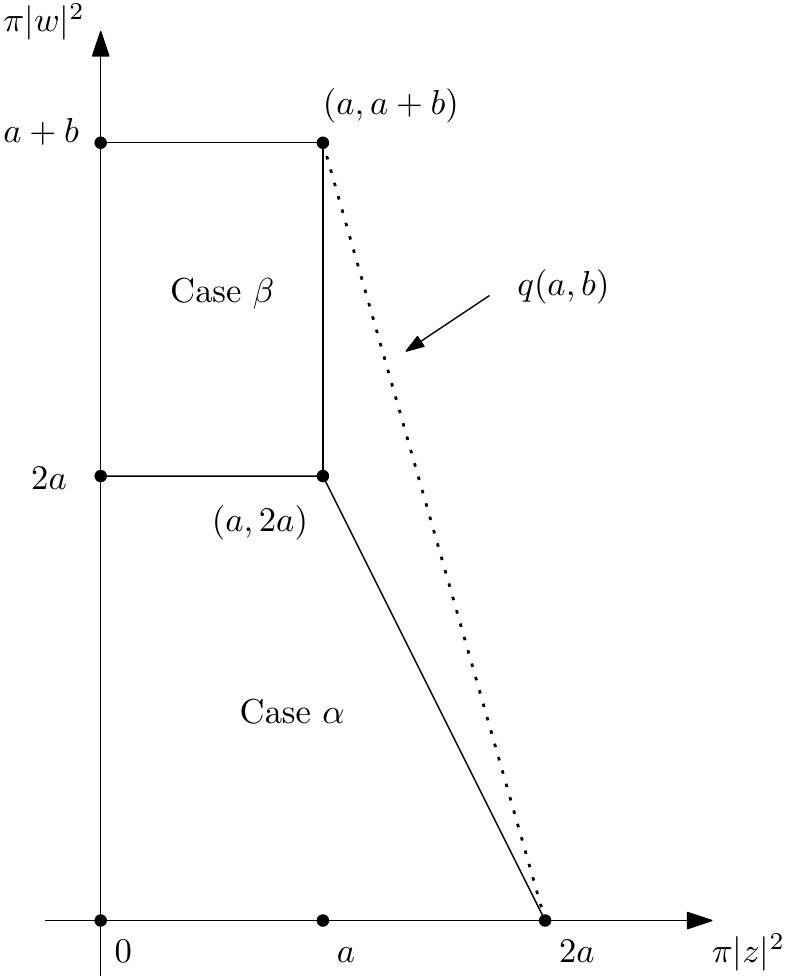}
\end{equation*}
The union of these regions is contained in the quadrilateral region defined by $q(a,b) \subset \R^2$ in our hypothesis. 
We note in ${\rm Case}\,\beta$ the points $\psi_1(z,w)$ lie in an $\eps$ neighborhood of $D(a) \times [0,1]$. Therefore the image under $\mathds{1} \times \Phi$ lies in $D(a) \times D(2a)$ and the moment image is described only by the lower trapezium in the graph above.
In other words, at $t=1$ the moment image of the Lagrangian isotopy (\ref{isotopy}) lies in an arbitrarily small neighborhood of $q(a,a)$.
Hence, we complete the proof. \end{proof}

In constructing lifts, our strategy will be to divide $\gamma$ into segments, and on each segment lift either by inclusions or by the rolled up Lagrangian embeddings described in Theorem \ref{construction}. The following lemma says that a Hamiltonian isotopy, which is also provided by Theorem \ref{construction} is enough to piece these lifts together.
In other words we are free to concatenate paths with endpoints lying in the same path component of a fiber.

\begin{lemma}\label{matching} Let $\gamma: [0,2] \to \R^2$ and $L_t$, $0 \le t \le 1$ and $M_t$, $1 \le t \le 2$ be smooth families of Lagrangian tori in $\mathcal{L}(X)$ such that $\P(L_t) = \gamma(t)$ and $\P(M_t) = \gamma(t)$ and $L_1$ is Hamiltonian isotopic to $M_1$ in $X$, that is, $L_1$ and $M_1$ lie in the same component of the fiber over $\gamma(1)$. Then $\gamma(t)$ has a smooth lift $N_t$ with $N_t = L_t$ for $t<1-\epsilon$ and $N_t = M_t$ for $t>1+\epsilon$.
\end{lemma}

\begin{proof} Let $\Phi_t \in {\rm Ham}(X)$ be a Hamiltonian flow with $\Phi_0 = \rm{I}$ and $\Phi_1(L_1) = M_1$.
First we replace $L_t$ by $\tilde{L}_t = \Phi_{f(t)}(L_t)$ where $f(t)=0$ for $t<1-\epsilon$ and $f(1)=1$. Then $\tilde{L}_t$ and $M_t$ together give a continuous lift of $\gamma(t)$.

To smooth a possible corner at $\tilde{L}_1 = M_1$ we first identify a neighborhood of $M_1$ in $X$ with a neighborhood of the zero section in $T^* \T^2$. Then Lagrangian tori near $M_1$ can be identified with the graphs of $1$ closed forms $\alpha = r d\theta_1 + sd\theta_2 + dg$ where $\theta_1$, $\theta_2$ are coordinates on $\T^2$ and $g(\theta_1, \theta_2)$ is a smooth function. Here $r$ and $s$ are uniquely defined, and $g$ is also uniquely defined if we insist that $\int g =0$. Moreover, we may assume that up to a translation the map $\P$ is given by $\P({\rm{gr}}(\alpha)) = (r,s)$. Now finding a smooth lift just requires replacing the family of $g$ corresponding to $\tilde{L}_t$ and $M_t$ by a smooth family of functions.
\end{proof}



Let $X_{\Omega}$ be a toric domain in $\R^4$ with moment image $\Omega$. Recall that a Type-I path is a path that entirely lies in $\Omega^+: = \Omega \cap \{r \leq s\}$ and any other path with the starting point in $\Omega^+$ is a Type-II path. The following corollary of Theorem \ref{construction} provides a useful sufficient condition to lift a general path. 

\begin{cor} [Path lifting criterion] \label{thm-lift} Let $X_{\Omega} \subset \R^4$ be a toric domain and let $\gamma = \{(r_t, s_t)\,|\, r_t < s_t\}_{t \in [0,1]}$ be an oriented path in ${\rm Sh}^+_H(X_{\Omega})$. Then $\gamma$ lifts to $\mathcal L(X_{\Omega})$ if it satisfies the following property: there exists a concatenation of sub-paths to form $\gamma$ such that any sub-path is either Type-I or is Type-II on an interval $[t_0, t_1]$ and satisfies the following conditions for some $t_* \in (t_0, t_1]$:
\begin{itemize}
\item[(II-i)] $\gamma|_{[t_0, t_*]} \in \Omega^{+}$ but $\gamma|_{[t_*,t_1]} \not\subset \Omega^+$,
\item[(II-ii)] $q(r_{t_*}, s_{t_*}) \subset {\rm int}(\Omega)$, 
\item[(II-iii)] $q(r_t, r_t) \subset {\rm int}(\Omega)$ for any $t \in [t_*, t_1]$;
\end{itemize}
moreover, if a Type-II sub-path concatenates further with a Type-I sub-path, then the condition (II-ii) above strengthens to hold also for $t = t_1$. \end{cor}

\begin{proof} [Proof of Corollary \ref{thm-lift}] By Example \ref{image-path}, any Type-I sub-path and the part $\gamma|_{[0, t_*]}$ of a Type-II sub-path always lifts via product tori, so we need to consider the part $\gamma|_{[t_*, t_1]}$ of a Type-II sub-path, which is not entirely contained in $\Omega^+$. But these segments lift using the family of rolled up embeddings described by Theorem \ref{construction} by condition (II-iii). To do this, starting from the Lagrangian $\psi_1(\phi_H^1(L(a,b)))$ as in Figure \ref{figure_more_wrap}, we can continuously change the areas of the disk $D(a)$ in the $z$-plane to be the given $r_t$ and the height $b$ of the rectangle in Figure \ref{figure_proj_z} to be the given $s_t$, for $t \in [t_*, t_1]$. The construction still applies since $r_t < s_t$ for all $t \in [0,1]$, and the resulting Lagrangian $L_t$ is Hamiltonian isotopic to the product Lagrangian torus $L(r_t, s_t)$ since we can run the constructions $\psi_t$ and $\phi_H^t$ above in reverse (cf.~the requirement of the condition $a<b$ for the constructions above). Moreover, the moment image of $L_t$ lies in an arbitrarily small neighborhood of $Q(r_t, r_t)$

It remains to adjust things so that our Lagrangian embeddings match at the endpoints, say $t_1$, of the segments, and by Lemma \ref{matching} it is enough to show these embeddings are Hamiltonian isotopic in $X_{\Omega}$. But Theorem \ref{construction} gives such an isotopy with support in $Q(r_{t_1}, s_{t_1})$, and by condition (II-ii) and its strengthening we have that $Q(r_{t_1}, s_{t_1}) \subset X_{\Omega}$. In this way, we get the desired conclusion. \end{proof}

\subsection{Proof of (II) in Theorem \ref{thm-path-lift-ball}, \ref{thm-path-lift-ellipsoid}, and \ref{thm-path-lift-polydisk}} \label{ssec-path-lift-ball-ellipsoid_2} 

\subsubsection{Proof of (II) in Theorem \ref{thm-path-lift-ball}} Recall that the moment image $\mu(B^4(R))$ is the right triangle $\Delta(R,R)$, intersecting $\R_{\geq 0}^2$ (in coordinate $(r,s)$) with points $(0, R)$ and $(R, 0)$. Condition (II-ii) in Corollary \ref{thm-lift}, $q(r_{t_*}, s_{t_*}) \subset {\rm int}(\Delta(R,R))$, is equivalent to the condition that vertex $(r_{t_*}, r_{t_*}+s_{t_*})$ lies strictly below the hypothenuse of $\Delta(R,R)$. Since the equation of the hypothenuse is $r+s = R$, the condition is
\[ r_{t_*} + (r_{t_*} + s_{t_*})<  R. \]
Similarly, Condition (II-iii) in Corollary \ref{thm-lift} implies $r_t < \frac{R}{3}$ for $t \in [0, T]$. 
Thus we complete the proof. \qed

\subsubsection{Proof of (II) in Theorem \ref{thm-path-lift-ellipsoid}} Recall that $\mu(E(a,b))$ is the right triangle $\Delta(a,b)$, intersecting $\R_{\geq 0}^2$ (in coordinate $(r,s)$) with points $(0, b)$ and $(a, 0)$. We have two cases in terms of the condition (II-ii) of Corollary \ref{thm-lift}, due to the comparison between the slope of the hypothenuse of this triangle and the slope of the hypothenuse of the quadrilateral region $q(r_{t_*}, s_{t_*})$ (see Figure \ref{figure_slopes}). 
\begin{figure}[h]
  \centering
   \includegraphics[scale=0.85]{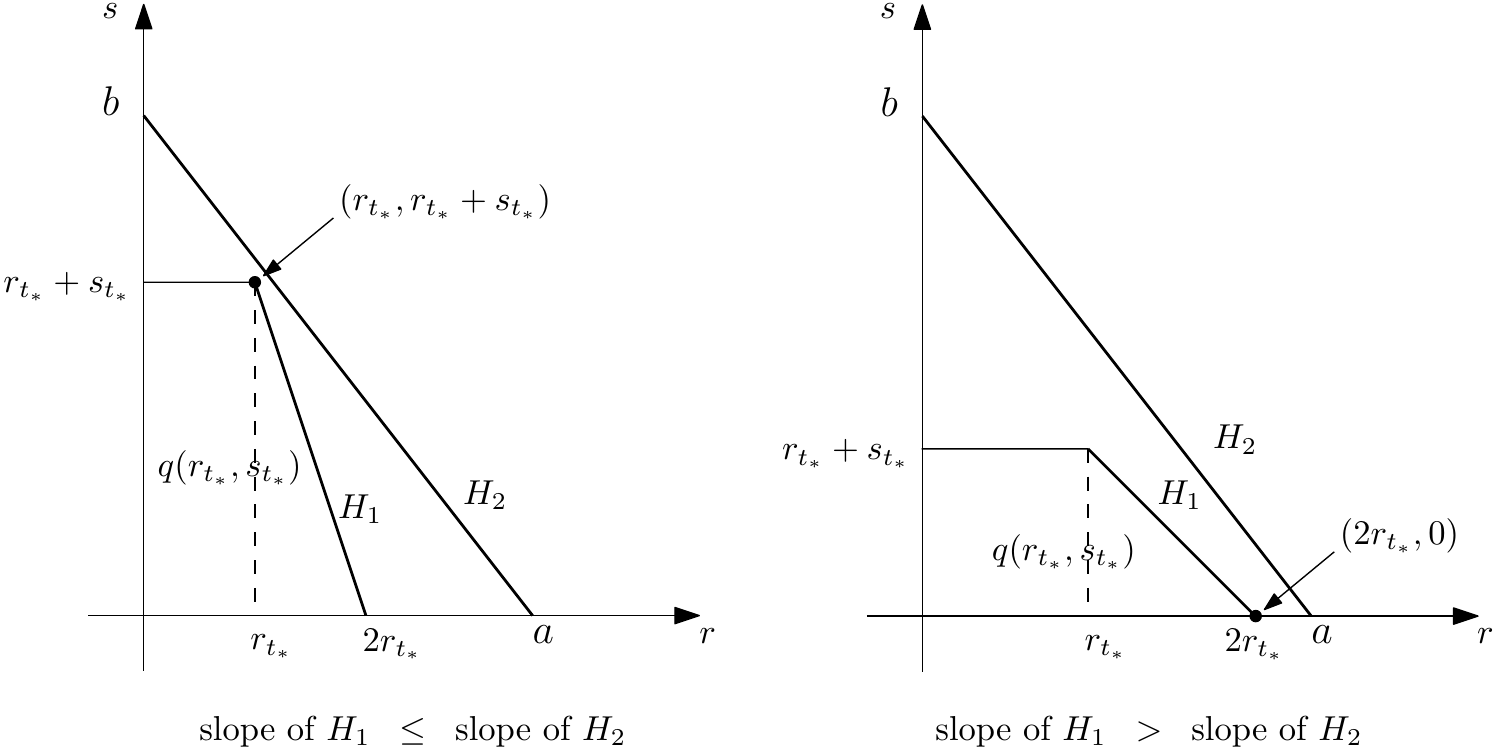} 
     \caption{Comparisons of slopes.} \label{figure_slopes}
\end{figure}
The slope $H_1$ is $- \frac{r_{t_*} + s_{t_*}}{r_*}$ while the slope of $H_2$ is $-k = \frac{b}{a}$. Then for the first case where $- \frac{r_{t_*} + s_{t_*}}{r_*} \leq -k$ (which is equivalent to $(k-1) r_{t_*} \leq s_{t_*}$), we have $q(r_{t_*}, s_{t_*}) \subset {\rm int}(\Delta(a,b))$ if and only if the vertex $(r_{t_*}, r_{t_*}+s_{t_*})$ lies strictly below the hypothenuse $H_2$. Since the line represented by $H_2$ is $\frac{r}{a} + \frac{s}{b} = 1$, we get
\[ r_{t_*} + s_{t_*} < s|_{r = r_*} = \left( 1- \frac{r_*}{a} \right) \cdot b = b - kr_*\]
which is equivalent to $(k+1)r_{t_*} + s_{t_*} < b$ as in (II-ii-1). Similarly, for the second case where $- \frac{r_{t_*} + s_{t_*}}{r_*} > -k$ (which is equivalent to $(k-1) r_{t_*} > s_{t_*}$), we have $q(r_{t_*}, s_{t_*}) \subset{\rm int}(\Delta(a,b))$ if and only if the vertex $(2r_*, 0)$ lies strictly on the left of $(a,0)$, that is, $2r_{t_*} < a$ as in (II-ii-2). 

In terms of the condition (II-iii) in Corollary \ref{thm-lift}, since the slope of the hypothenuse $H_1$ of the quadrilateral region $q(r_{t}, r_{t})$ is always $-2$ which is no greater than the slope of the hypothenuse $H_2$ (since we assume $k \geq 2$). Therefore the second case above is enough to show that (ii-iii) is implied by $2r_t < a$. 
Thus we complete the proof.\qed

\subsubsection{Proof of (II) in Theorem \ref{thm-path-lift-polydisk}} Recall that the moment image $\mu(P(c,d))$ is the rectangle $\Box(c,d)$ with intercepts $(c,0)$ and $(0,d)$. Condition (II-ii) in Corollary \ref{thm-lift}, $q(r_{t_*}, s_{t_*}) \subset {\rm int}(\Box(a,b))$, is equivalent to the vertex $(2r_*, 0)$ lying on the left of $(c,0)$ (cf.~the condition (II-ii-2) in the proof of Theorem \ref{thm-path-lift-ellipsoid}) and the height $r_* + s_*$ lower than $d$. In other words, we have 
\[ r_* < \frac{c}{2} \,\,\,\,\mbox{and}\,\,\,\, r_* + s_* <d. \]
Similarly, Condition (II-iii), $q(r_{t}, r_{t}) \subset {\rm int}(\Box(c,d))$, is determined by the vertex $(2r_t, 0)$, which implies $r_t < \frac{c}{2}$ for $t \in [0,T]$. Thus we complete the proof. \qed

\subsection{Proof of Theorem \ref{knotted-ball}, \ref{knotted-ellipsoid}, and \ref{knotted-polydisk}} \label{ssec-thm-knotted-ball-ellipsoid} Now, we are ready to see how these path lifting criterions easily imply the existence of knotted Lagrangian tori in $B^4(R)$, $E(a,b)$, and $P(c,d)$. We will give the proof of these three theorems simultaneously. 

For any area classes $(r,s)$ in the given region (\ref{knotted-ball-1}) or (\ref{knotted-ellipsoid-1}) or (\ref{knotted-polydisk-1}), consider the path $\gamma = \{\gamma(t)\}_{t \in [0,1]}$ with
\[ \gamma(t) = (r, (1-t)s + ts_*)\] 
where $(r, s_*) \notin \Delta(R,R)$ or $(r, s_*) \notin \Delta(a,b)$ or $(r, s_*) \notin \Box(c,d)$,  respectively. By Theorem \ref{shapecalc}, we know $\gamma \subset {\rm Sh}_H^+(B^4(R))$, $\gamma \subset {\rm Sh}_H^+(E(a,b))$, and $\gamma \subset {\rm Sh}_H^+(P(c,d))$, respectively. Choose any other path $\tilde{\gamma}$ that sits entirely inside $\Delta(R,R)^+$ or $\Delta(a,b)^+$ or $\Box(c,d)^+$ such that the whole path satisfies (II-iii), starts at a point satisfying (II-ii) and ends at $(r,s_*)$. Now, consider the concatenation 
\[ \tilde{\gamma} \# \gamma \in {\rm Sh}_H^+(B^4(R)) \,\,\mbox{or}\,\, {\rm Sh}_H^+(E(a,b))\,\,\mbox{or}\,\, {\rm Sh}_H^+(P(c,d)). \]
By (II) in Theorem \ref{thm-path-lift-ball}, Theorem \ref{thm-path-lift-ellipsoid}, and Theorem \ref{thm-path-lift-polydisk}, respectively, we know $\tilde{\gamma} \# \gamma$ lifts to a Lagrangian isotopy of tori. In particular, there exists a Lagrangian sub-isotopy which projects via $\P$ to $\gamma$, starting from $(r,s)$. If there do not exist any knotted Lagrangian tori in the fiber $\P^{-1}((r,s))$, then up to Hamiltonian isotopy in $B^4(R)$ or $E(a,b)$ or $P(c,d)$, we can assume this sub-isotopy starts from the product Lagrangian torus $L(r,s)$. Then by Definition \ref{dfn-path-lift}, $\gamma$ lifts and it contradicts (I) in  Theorem \ref{thm-path-lift-ball}, Theorem \ref{thm-path-lift-ellipsoid}, and Theorem \ref{thm-path-lift-polydisk}, respectively. Thus we complete the proof. \qed

\section{An exotic example} \label{sec-ex-exotic}

As explained in the introduction, the path lifting can be subtle due to various reasons. In this section, we illustrate this complexity via an exotic example. Consider $\gamma_1$ and $\gamma_2$ in ${\rm Sh}_H^+(E(a,b))$ with $k: = \frac{b}{a} = 3$ shown in Figure \ref{figure-ex-exotic}.
\begin{figure}[h]
  \centering
      \includegraphics[scale=0.85]{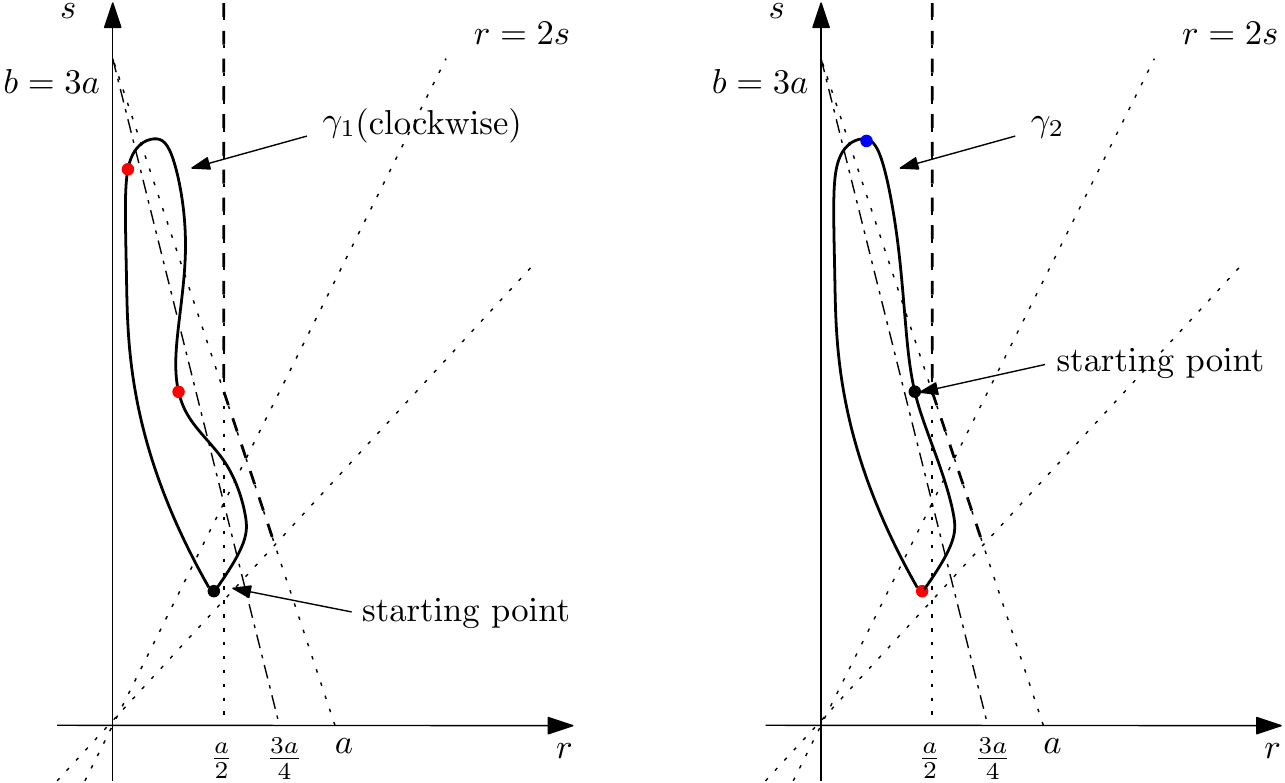} 
  \caption{$\gamma_1$ lifts, and $\gamma_2$ lifts only viewed clockwise.} \label{figure-ex-exotic}
\end{figure}
In the left picture, $\gamma_1$ when viewed clockwise lifts to ${\rm Sh}_H^+(E(a,b))$. The picture indicates the concatenation points (red), and the sub-paths change types from Type-I to Type-II, then back to Type-I (see section \ref{ssec-path-lifting} ).

On the other hand, in the right picture, the path $\gamma_2$ is a small perturbation of $\gamma_1$ as a geometric path. If we view it clockwise, then this is an example showing a certain {\it monodromy} phenomenon. Note that $\gamma_2$ lifts by (II) in Theorem \ref{thm-path-lift-ellipsoid} and the picture indicates the concatenation point. Explicitly, it starts from a Type-I sub-path, followed by a Type-II sub-path. This Type-II sub-path ends at the same point as the starting point. Meanwhile, since this Type-II sub-path never comes back to the ``flexible'' region below the line $4r +s = b$, as in the left picture, the Lagrangian torus at the endpoint is not a product torus. Therefore, this closed loop $\gamma_2$ lifts but not to a {\it loop} of Lagrangian tori. 

Finally, if we view $\gamma_2$ counterclockwise, i.e., considering its reverse path $\overline{\gamma}_2$,  then it does {\it not} lift. Otherwise, the sub-path from the starting point to the blue dot will lift, and it violates the obstruction, i.e., (I) in Theorem \ref{thm-path-lift-ellipsoid}. This shows that the orientation in the path lifting also matters. 

\section{Obstructions to symplectic embeddings}

In this section, we will demonstrate how to use path lifting to obstruct symplectic embeddings between domains in $\R^4$. Let us start from the following result. 

 \begin{prop} \label{prop-obs-domain} Let $X, Y$ be two toric domains in $\R^4$, and $\gamma$ be a path in ${\rm Sh}_H^{+}(X)$ that lifts to a Lagrangian isotopy of tori in $\mathcal L(X)$, denoted by $\{L_t\}_{t \in [0, T]}$. If there exists a symplectic embedding $X \hookrightarrow Y$ such that the image of $L_0$ is unknotted in $Y$, then $\gamma$ lies in ${\rm Sh}_H^{+}(Y)$ and lifts to $\mathcal L(Y)$. \end{prop}

\begin{proof} The first conclusion directly comes from Proposition 7.1 in \cite{hindzhang}. It suffices to prove the second conclusion. Suppose $\gamma = \{(r_t, s_t) \in {\rm Sh}_H^+(X) \,|\, t \in [0,T]\}$. Since $\gamma$ lifts to $\mathcal L(X)$, the corresponding Lagrangian isotopy of tori $L = \{L_t\}_{t \in [0, T]}$ satisfies $\P(L_t) = (r_t, s_t)$. Suppose the symplectic embedding $X \hookrightarrow Y$ is $\phi$, and consider the Lagrangian isotopy 
\[ \phi(L) = \{\phi(L_t)\}_{t \in [0, T]} \subset \mathcal L(Y). \]
Similarly to Proposition 7.1 in \cite{hindzhang}, we know $\P(\phi(L_t)) = (r_t, s_t)$ for any $t \in [0, T]$. Meanwhile, by assumption, $\phi(L_0)$ is unknotted in $Y$, so by definition there exists a Hamiltonian isotopy $\Psi = \{\Psi_t\}_{t \in [0,T]}$ of $Y$ such that $\Psi_1(\phi(L_0)) = L(r_0, s_0)$. Then the following Lagrangian isotopy of tori in $Y$, 
\[ \tilde{L} = \{\Psi_{1-t}(\phi(L_t))\}_{t \in [0,T]} \subset \mathcal L(Y)\]
is the desired lift of $\gamma$ in ${\rm Sh}_H^+(Y)$ since ${\rm Ham}(Y)$ preserves the fiber of $\P$. In other words, $\gamma \subset {\rm Sh}_H^+(Y)$ lifts to $\mathcal L(Y)$ by $\tilde{L}$, and we complete the proof. \end{proof}

\subsection{Proof of Theorem \ref{emb-knotted}}\label{obs-domain_1} We will only give the proof of (1), and (2), (3) can be proved in a similar manner.  

For the given $(r,s)$, consider the {\it straight line} path $\gamma$ starting at $(r,s)$ and ending at $(0,x)$. Since it is a straight line with $2r+s >R$ and $x>R$, this path $\gamma$ lie entire outside the ``flexible'' region below the line $2r+s =R$. Meanwhile, again $x>R$ implies that $\gamma$ will escape the moment image $\Delta(R,R)$ eventually. Then on the one hand, since $\gamma$ lies entirely in the moment image $\Delta(1,x)$, Example \ref{image-path} implies that it lifts to $\mathcal L(E(1,x))$. On the other hand, if $\phi(L(r, s))$ is unknotted under the embedding $\phi: E(1,x) \hookrightarrow B^4(R)$, then Proposition \ref{prop-obs-domain} implies that $\gamma \subset {\rm Sh}_H^+(B^4(R))$ and it lifts as well. However, this contradicts to (I) in Theorem \ref{thm-path-lift-ball}. Therefore, we obtain the desired conclusion. \qed. 

\subsection{Proof of Theorem \ref{thm-obs-toric}} \label{obs-domain_2} The assumption (ii) implies that $\gamma$ lifts to $\mathcal L(X)$ since $\gamma$ lies entirely in $\mu(X)^+$. In particular, the starting point $(r_0, s_0)$ can be realized as the inclusion of the product Lagrangian torus $L(r_0, s_0) \hookrightarrow X$. Now, suppose that there exists a symplectic embedding $\phi: X \hookrightarrow E(a,b)$. By the first condition in the assumption (i), we have
\[ L(r_0, s_0) \subset E \xhookrightarrow{i} X \xhookrightarrow{\phi} E(a,b), \]
where $i$ is the inclusion. Since Corollary 1.6 in \cite{McD09} proves that the space of symplectic embeddings from $E$ to $E(a,b)$ denoted by ${\rm Emb}(E, E(a,b))$ is connected, the condition $E \subset E(a,b)$ implies that every symplectic embedding from $E$ to $E(a,b)$ is isotopic to the inclusion. In particular, the product Lagrangian torus $L(r_0,s_0) \subset E$ is unknotted under the symplectic embedding $\phi \circ i$. Then Proposition \ref{prop-obs-domain} implies that $\gamma \subset {\rm Sh}_H^+(E(a,b))$ and it lifts to $\mathcal L(E(a,b))$. This contradicts (I) in Theorem \ref{thm-path-lift-ellipsoid}. Therefore, we obtain the desired contradiction. \begin{figure}[h]
  \centering
      \includegraphics[scale=0.9]{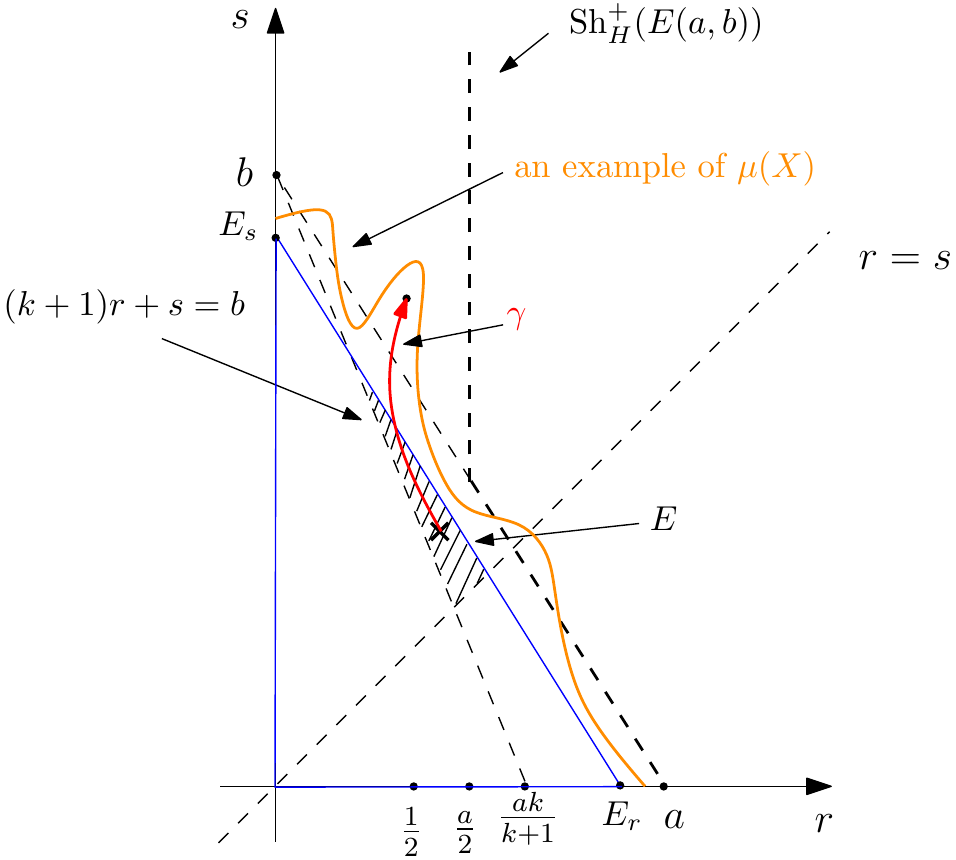}
  \caption{The obstruction is given by the red line segment.} \label{figure_obs_3}
\end{figure}
Figure \ref{figure_obs_3} shows a schematic picture of this proof, where $E = E_{\Delta(E_r, E_s)}$ and the moment image $\Delta(E_r, E_s)$ is shown as the blue triangle, $\mu(X)$ is described via the orange curve, and the obstruction is given by $\gamma$, shown as the red oriented path.  \qed 

\section{Appendix} \label{app} 
This section verifies the following result, up to a rescaling by $\frac{a}{k+1}$, that was used in (2) in Example \ref{emb-knotted-ex}. 

\begin{prop} \label{prop-ellipsoid-emb} For any $k \in \N$, $E(k, (k+1)^2) \hookrightarrow E(k+1, k(k+1))$. \end{prop}

From Hutchings' work \cite{Hut11} and McDuff's work \cite{McD11}, there exists a complete characterization of four-dimensional ellipsoid embeddings, that is, 
\begin{equation} \label{ellipsoid-emb}
E(c,d) \hookrightarrow E(a,b) \,\,\,\,\,\mbox{if and only if} \,\,\,\,\, \mathcal N(c,d)_k \leq \mathcal N(a,b)_k.
\end{equation}
Here, $\mathcal N(a,b)$ denotes an infinite sequence of numbers consisting of all the non-negative linear combination $ma + nb$ (for $m, n \in \Z_{\geq 0}$) in a non-decreasing order (with repetitions), and $\mathcal N(a,b)_k$ is the $k$-th entry in $\mathcal N(a,b)$, similarly to $\mathcal N(c,d)_k$ and $\mathcal N(c,d)$. In fact, there exists a nice geometric description of $\mathcal N(a,b)$ in general (see Section 3.3 in \cite{Hut11}). Denote by $\Delta_{a,b}(t)$ the closed right triangle in $\R^2$ with vertices $(0,0)$, $(\frac{t}{a}, 0)$ and $(0, \frac{t}{b})$ for $t \geq 0$, and denote 
\begin{equation} \label{dfn-lattice}
{\mathcal R}_{a,b}(t) := \#\{\Delta_{a,b}(t) \cap \Z_{\geq 0}^2\}. 
\end{equation}
Then the characterization (\ref{ellipsoid-emb}) is equivalent to the statement that $E(c,d) \hookrightarrow E(a,b)$ if and only if $t_2 \leq t_1$ whenever $\mathcal R_{c,d}(t_2) = \mathcal R_{a,b}(t_1)$. Since $\mathcal R_{a,b}(t)$ and $\mathcal R_{c,d}(t)$ are non-decreasing functions of $t$, in return, (\ref{ellipsoid-emb}) is further equivalent to the following statement,  
\begin{equation} \label{equiv-lattice}
E(c,d) \hookrightarrow E(a,b) \,\,\,\,\mbox{if and only if} \,\,\,\, \mathcal R_{c,d}(t) \geq \mathcal R_{a,b}(t) \,\,\mbox{for all $t \geq 0$}. 
\end{equation}
We emphasize that the statement (\ref{ellipsoid-emb}) (as well as (\ref{equiv-lattice})), in general, is not always easy to verify. However, the equivalent statement (\ref{equiv-lattice}) has the advantage that some elementary geometry propositions can be applied. For instance, in terms of counting lattice points, the well-known Pick's theorem is very useful. Explicitly, for any polygon with integer vertices and without holes, 
\begin{equation} \label{pick}
\#\left\{\mbox{interior lattice points}\right\}+ \frac{\#\left\{\mbox{boundary lattice points}\right\}}{2}  = {\mbox{area}} + 1. 
\end{equation}
The proof of Proposition  \ref{prop-ellipsoid-emb} turns out to be a nice combination of the criterion (\ref{equiv-lattice}) and Pick's theorem (\ref{pick}). 

\begin{proof} [Proof of Proposition  \ref{prop-ellipsoid-emb}] By (\ref{equiv-lattice}), it suffices to show ${\cal R}_{k,(k+1)^2}(t) \geq  {\cal R}_{k+1,k(k+1)}(t)$
for any $t \geq 0$. We will prove this in two steps. First, we have the following result. 

\begin{lemma} \label{lemma-integer} For any $k \in \N$ and $A \in \Z_{\geq 0}$, we have 
\[ {\cal R}_{k,(k+1)^2}(Ak(k+1)) =  {\cal R}_{k+1,k(k+1)}(Ak(k+1)) = \frac{1}{2} kA(A+1) + (A+1).\]
 \end{lemma}

\begin{proof}[Proof of Lemma \ref{lemma-integer}] 
For the right triangle ${\Delta}_{k+1,k(k+1)}(Ak(k+1))$, its $x$-intercept is $\frac{t}{a} = \frac{Ak(k+1)}{k+1} = Ak$ and its $y$-intercept is $\frac{t}{b} = \frac{Ak(k+1)}{k(k+1)} = A$, where both $Ak, k \in \Z_{\geq 0}$. Then Pick's Theorem applies, and (\ref{pick}) implies that 
\[ {\cal R}_{k+1,k(k+1)}(Ak(k+1)) = \frac{A^2k}{2} + 1 + \frac{\#\left\{\mbox{boundary lattice points}\right\}}{2}. \]
Meanwhile, by elementary counting, $\#\left\{\mbox{boundary lattice points}\right\} = 2A + Ak$. Therefore, we get the desired conclusion for ${\cal R}_{k+1,k(k+1)}(Ak(k+1))$. 

\medskip

For the right triangle ${\Delta}_{k,(k+1)^2}(Ak(k+1))$, let us introduce the following notations, 
\[ c = \floor*{\frac{Ak}{k+1}} = \frac{Ak}{k+1} - \frac{C}{k+1}\]
where $C \in \N$ and $0 \leq C \leq k$, and
\[ b = \floor*{\frac{A}{k+1}}= \frac{A}{k+1} - \frac{B}{k+1}\]
where $B \in \N$ and $0 \leq B \leq k$. Here is a useful observation. 

\begin{claim} \label{claim-1} Either $B=C=0$ or $B + C = k+1$. In particular, $\floor*{\frac{Bk}{k+1}} = B-1$. \end{claim}

\begin{proof} [Proof of Claim \ref{claim-1}] We have $$\frac{Ak}{k+1} = kb + \frac{kB}{k+1}.$$ Therefore we can write $kB = m(k+1) + C$ where $m \ge 0$ is an integer. More explicitly, $m = \floor*{\frac{kB}{k+1}}$. As $C \ge 0$ we have $m <B$, but if $m \le B-2$ then $$C = kB - m(k+1) \ge kB - (k+1)(B-2) = 2k - B + 2 \ge k+2$$ a contradiction. Hence $m = B-1$ and we get $C = kB - (k+1)(B-1) = k+1 - B.$ Therefore, we get the desired claim. \end{proof}

Next, we will count lattice points in $\Delta_{k, (k+1)^2}(Ak(k+1))$ in the following way. Divide ${\Delta}_{k,(k+1)^2}(Ak(k+1))$ into two parts as in Figure \ref{figure-cut}, one small triangle $\Delta_{\rm small}$ and one trapezoid $P_{\rm trapezoid}$. 
\begin{figure}[h] 
\includegraphics[scale=0.8]{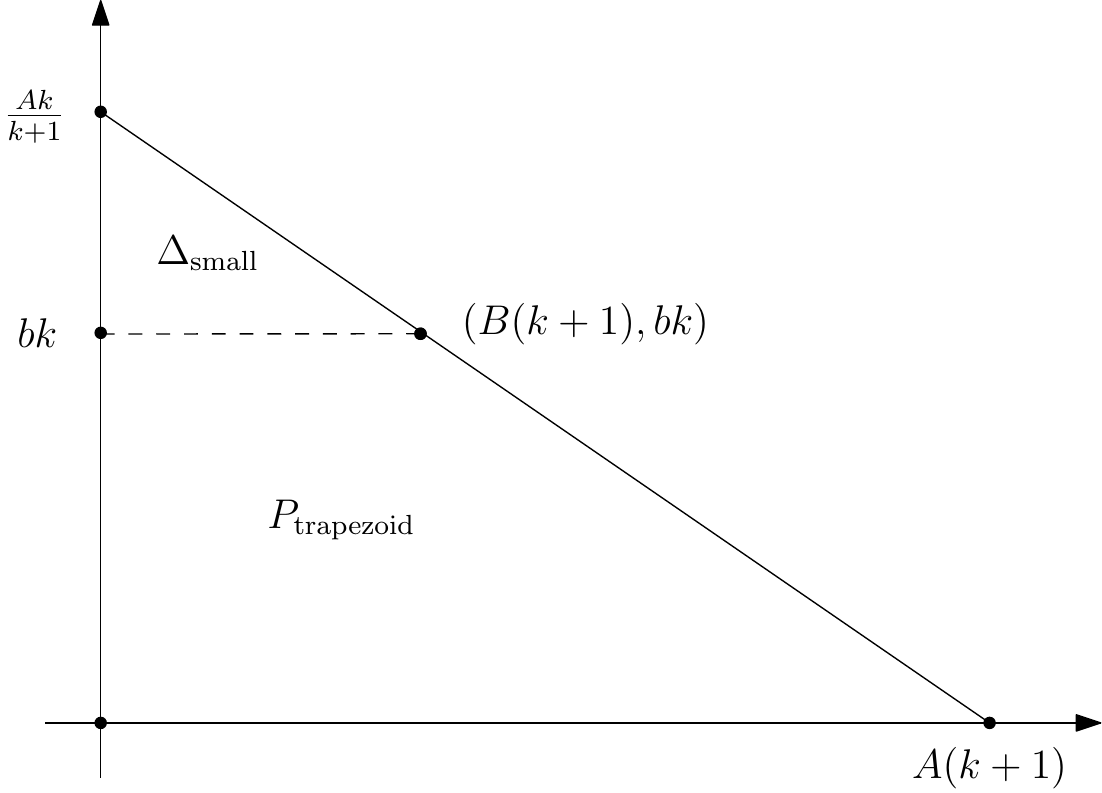}
\caption{A division of ${\Delta}_{k,(k+1)^2}(Ak(k+1))$.}\label{figure-cut}
\end{figure}
The way we cut in Figure \ref{figure-cut} guarantees that there are no lattice points on the hypothenuse of $\Delta_{\small}$ except possibly at the endpoints. Indeed, since the slope of ${\Delta}_{k,(k+1)^2}(Ak(k+1))$ is $- \frac{k}{(k+1)^2}$, moving from the vertex $(A(k+1), 0)$ to the vertex $(0,0)$, only the multiples of $(k+1)^2$ can generate a lattice point on the hypothenuse of ${\Delta}_{k,(k+1)^2}(Ak(k+1))$. There are at most $\floor*{\frac{A}{k+1}}$-many non-zero multiples of $(k+1)^2$ in the interval $[0, A(k+1)]$ and the smallest one is 
\[ A(k+1) - \floor*{\frac{A}{k+1}} (k+1)^2 = \left(\frac{A}{k+1}-\floor*{\frac{A}{k+1}}\right) (k+1)^2 = B(k+1). \]
Meanwhile, it is easy to verify that the corresponding $y$-coordinate is $bk$. In particular, $P_{\rm trapezoid}$ is a polygon with integer vertices and without holes. Then Pick's theorem applies, and (\ref{pick}) implies that 
\[ \# \left\{\begin{array}{c} \mbox{lattice points} \\ \mbox{in $P_{\rm trapezoid}$} \end{array}\right\} = \frac{bk(k+1)(A+B)}{2} +1 +  \frac{\#\left\{\mbox{boundary lattice points}\right\}}{2}. \]
Moreover, by an elementary counting, $\#\{\mbox{boundary lattice points}\} = (A+B+b)(k+1)$. Therefore, 
\begin{equation} \label{lattice-trap}
\# \left\{\begin{array}{c} \mbox{lattice points} \\ \mbox{in $P_{\rm trapezoid}$} \end{array}\right\}  = \frac{bk(k+1)(A+B)}{2} +1  + \frac{(A+B+b)(k+1)}{2}. 
\end{equation}
For the small triangle $\Delta_{\rm small}$, up to an integer shift (explicitly shifted down by $bk$), it suffices to consider the following triangle in Figure \ref{figure-small}. 
\begin{figure}[h] 
\includegraphics[scale=0.8]{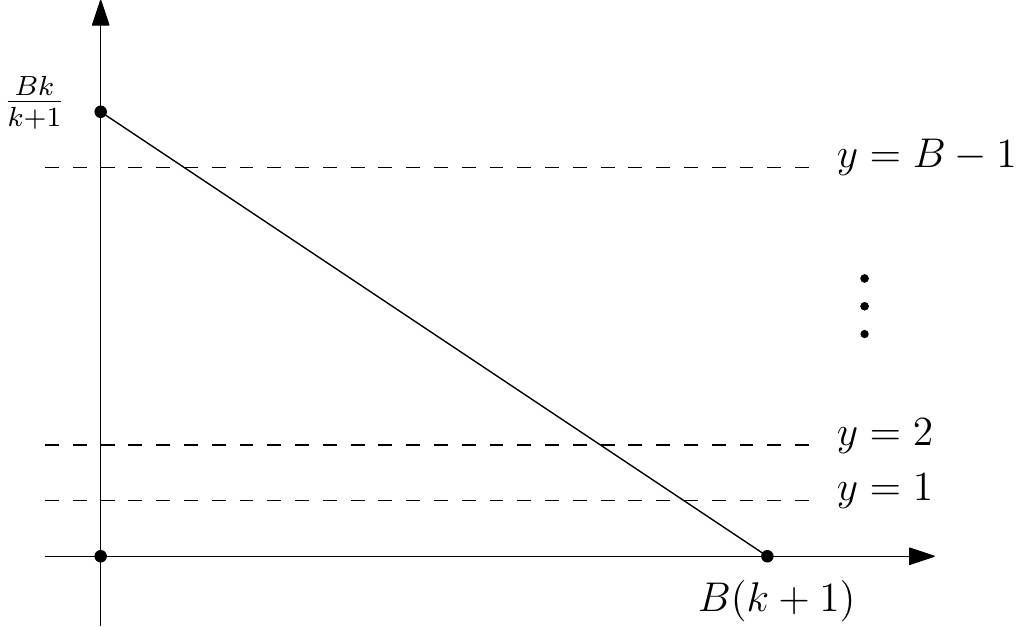}
\caption{Line-by-line counting in $\Delta_{\rm small}$.}\label{figure-small}
\end{figure}
In particular, the top line is $y = \floor*{\frac{Bk}{k+1}}$ and Claim \ref{claim-1} above says that $\floor*{\frac{Bk}{k+1}} = B-1$. Since there are no lattice points in the interior of the hypothenuse of the triangle in Figure \ref{figure-small}, for each $\ell \in \{1, …, B-1\}$, the intersection of the line $y = \ell$ with this triangle admits $(B(k+1) - \ell(k+2) -1)$-many lattice points in the interior. Then 
\begin{align}\label{lattice-tri}
\# \left\{\begin{array}{c} \mbox{lattice points} \\ \mbox{in $\Delta_{\rm small}$} \end{array}\right\} & = \sum_{\ell=1}^{B-1} (B(k+1) - \ell(k+2) -1) + (B-1) + B(k+1) \nonumber\\
& = B(B-1)(k+1) - \frac{B(B-1)(k+2)}{2} + B(k+1) \nonumber \\
& = \frac{B(B-1)k}{2} + \underbrace{B(k+1)}_{{\tiny \mbox{will be redundant}}}.
\end{align}
Hence, by summing up the lattices points in $P_{\rm trapezoid}$ (as in (\ref{lattice-trap})) and in $\Delta_{\rm small}$ (as in (\ref{lattice-tri})), we have
\begin{align*}
{\cal R}_{k,(k+1)^2}(Ak(k+1))& = \# \left\{\begin{array}{c} \mbox{lattice points} \\ \mbox{in $P_{\rm trapezoid}$} \end{array}\right\}  +  \#\left\{\begin{array}{c} \mbox{lattice points} \\ \mbox{in $\Delta_{\rm small}$} \end{array}\right\}  - \#\left\{\begin{array}{c} \mbox{repeated} \\ \mbox{counting}\end{array}\right\}\\
& =  \frac{bk(k+1)(A+B)}{2} +1  + \frac{(A+B+b)(k+1)}{2} + \frac{B(B-1)k}{2}\\
& = \frac{1}{2} kA(A+1) + (A+1)
\end{align*}
where the final step comes from a series of simplifications by using the relation $B = A - b(k+1)$. Thus we complete the proof of Lemma \ref{lemma-integer}.
\end{proof}

Suppose now that $t = Ak(k+1) + s$ where $0<s<k(k+1)$. Observe that the graph of ${\cal R}_{k+1,k(k+1)}(t)$ is horizontal with jumps when $s = a(k+1)$ for integers $a$. Thus we may assume $s=a(k+1)$ with $1 \le a \le k-1$. Then we have
\begin{equation} \label{lower-target}
{\cal R}_{k+1,k(k+1)}((Ak+a)(k+1)) - {\cal R}_{k+1,k(k+1)}(Ak(k+1)) = (A+1)a.
\end{equation}
It remains to estimate ${\cal R}_{k,(k+1)^2}((Ak+a)(k+1)) - {\cal R}_{k,(k+1)^2}(Ak(k+1))$, and we aim to obtain {\it at least} $(A+1)a$ as in (\ref{lower-target}) (see Figure \ref{figure-extended}). 
\begin{figure}[h] 
\includegraphics[scale=0.8]{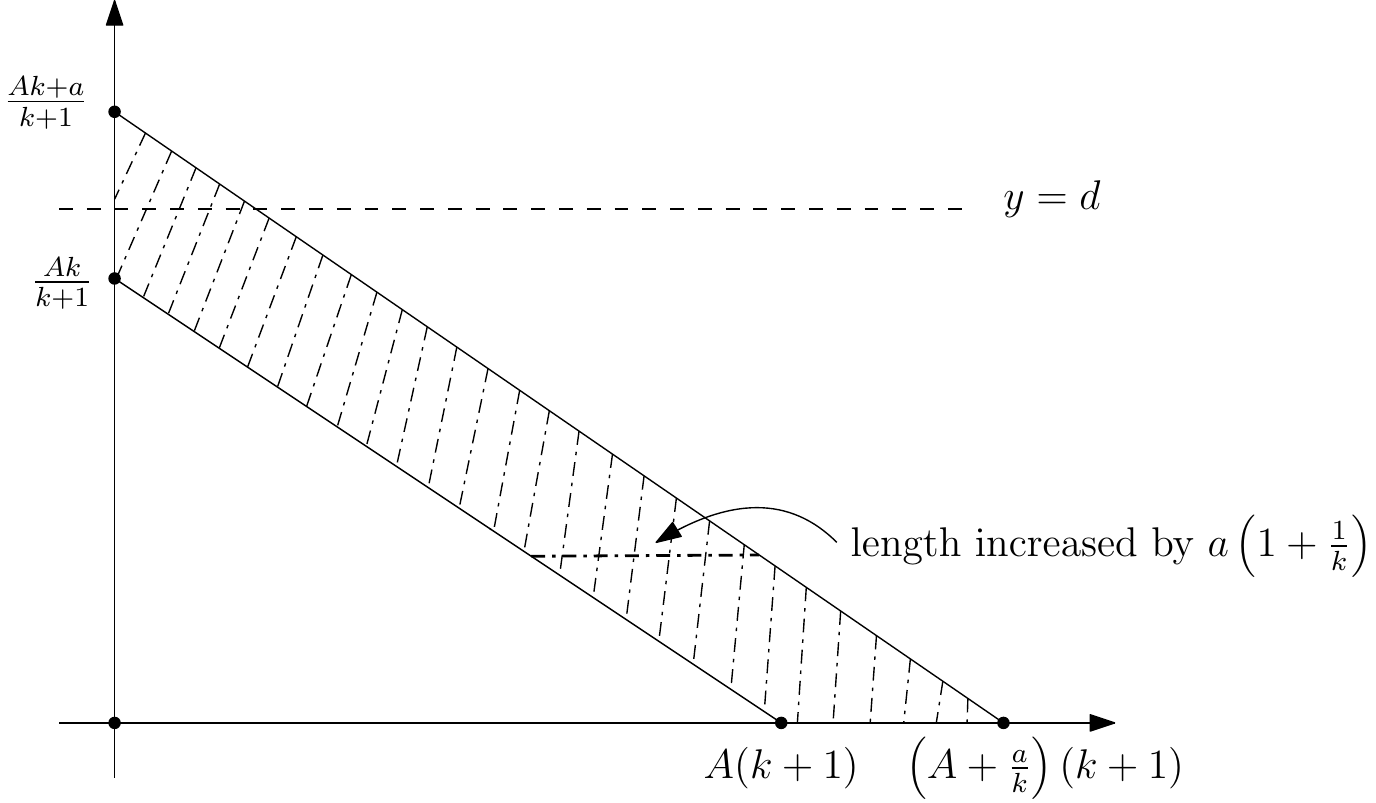}
\caption{Count additional lattice points in the shaded region.}\label{figure-extended}
\end{figure}
To this end, we use the same notations as above, but also introduce 
\[ d = \floor*{\frac{Ak+a}{k+1}} = \frac{Ak+a}{k+1} - \frac{D}{k+1}, \]
where $D \in \N$ and $0 \leq D \leq k$. Note that $d-c \le 1$. Moreover, 
\begin{itemize}
\item[(i)] if $d=c$, then $a \le k+1 - C$ and $D=C+a$; 
\item[(ii)] if $d=c+1$, then $D = C+a - (k+1)$. 
\end{itemize}
When (ii) is satisfied, the new triangle intersects $y=d$ and the row contains
\begin{equation} \label{extra-row}
1+ \floor*{\frac{D}{k+1}\frac{(k+1)^2}{k}} = 1 + \floor*{\frac{(k+1)D}{k}}
\end{equation}
lattice points. Meanwhile, the lengths of the other rows increase by $a(1 + \frac{1}{k})$, so  we count the number of additional lattice points as
$$a + (a+1) + \dots + (a+1) + a + \dots + a + \cdots, $$
where the $(a+1)$ terms come in blocks of length $a$ and the $a$ terms in blocks of length $k-a$. Therefore, depending whether the sum ends with $(a+1)$ (if $c-kb \le a$) or $a$ terms (if $c-kb \ge a+1$), we have 
\[ \#\left\{\begin{array}{c} \mbox{additional} \\ \mbox{points}\end{array}\right\}  =  \left\{ \begin{array}{lcl} a + (k+1)ab + (c-kb)(a+1) & \mbox{if} & c-kb \leq a \\ a+ (k+1)ab + a(c-kb +1) & \mbox{if} & c-kb \geq a+1 \end{array} \right..\]
It is easy to obtain the desired $(A+1)a$-many additional lattice points when $B = C =0$. If not, then Claim \ref{claim-1} implies that $b+c = A-1$ and $c-kb = B-1$, and a further simplification gives  
\[ \#\left\{\begin{array}{c} \mbox{additional} \\ \mbox{points}\end{array}\right\}  =  \left\{ \begin{array}{lcl} aA + B-1 & \mbox{if} & c-kb \leq a \\ (A+1)a & \mbox{if} & c-kb \geq a+1 \end{array} \right..\]
Therefore, it suffices to focus on the case where $c-kb \leq a$. 

\medskip

If $d=c$, then by the item (i) above, $D = C+ a = k+1 - B + a$, which implies that $B-1= a + k - D \ge a$. Thus we get at least $(A+1)a$-many additional lattice points as required. If $d = c+1$, then we have an extra row so the cardinality of the additional points in total is $aA + B + \floor*{\frac{(k+1)D}{k}}$ by (\ref{extra-row}). The item (ii) above implies that $D = a - B$ and then we have
\[ aA + B + \floor*{\frac{(k+1)D}{k}} = aA +a -D +\floor*{\frac{(k+1)D}{k}}  \ge (A+1)a. \]
Therefore, we complete the proof of Proposition \ref{prop-ellipsoid-emb}. \end{proof}

\begin{remark} Since ${\rm Vol}(E(k, k+1)^2) = {\rm Vol}(E(k+1, k(k+1))) = \frac{k(k+1)^2}{2}$, the symplectic embedding guaranteed by Proposition \ref{prop-ellipsoid-emb} is volume-filling. In other words, this symplectic embedding is optimal in the sense there does not exist any $\lambda >1$ such that $\lambda E(k, (k+1)^2) \hookrightarrow E(k+1, k(k+1))$. \end{remark}

\vspace*{5mm}

\bibliographystyle{amsplain}
\bibliography{biblio_hkpl}

\providecommand{\bysame}{\leavevmode\hbox to3em{\hrulefill}\thinspace}
\providecommand{\MR}{\relax\ifhmode\unskip\space\fi MR }
\providecommand{\MRhref}[2]{%
  \href{http://www.ams.org/mathscinet-getitem?mr=#1}{#2}
}
\providecommand{\href}[2]{#2}
\begin{thebibliography}{10}

\bibitem{BEHWZ03}
Fr{\'e}d{\'e}ric Bourgeois, Yakov Eliashberg, Helmut Hofer, Krzysztof Wysocki,
  and Eduard Zehnder, \emph{Compactness results in symplectic field theory},
  Geom. Topol. \textbf{7} (2003), 799--888. \MR{2026549}

\bibitem{BEP12}
Lev Buhovsky, Michael Entov, and Leonid Polterovich, \emph{Poisson brackets and
  symplectic invariants}, Selecta Math. (N.S.) \textbf{18} (2012), no.~1,
  89--157. \MR{2891862}

\bibitem{chaperon83}
Marc Chaperon, \emph{Quelques questions de g\'{e}om\'{e}trie symplectique},
  Ast\'{e}risque, vol. 105, pp.~231--249, Soc. Math. France, Paris, 1983.
  \MR{728991}

\bibitem{Che96}
Yu.~V. Chekanov, \emph{Lagrangian tori in a symplectic vector space and global
  symplectomorphisms}, Math. Z. \textbf{223} (1996), no.~4, 547--559.
  \MR{1421954}

\bibitem{CS10}
Yuri Chekanov and Felix Schlenk, \emph{Notes on monotone {L}agrangian twist
  tori}, Electron. Res. Announc. Math. Sci. \textbf{17} (2010), 104--121.
  \MR{2735030}

\bibitem{CM18}
K.~Cieliebak and K.~Mohnke, \emph{Punctured holomorphic curves and {L}agrangian
  embeddings}, Invent. Math. \textbf{212} (2018), no.~1, 213--295. \MR{3773793}

\bibitem{C-G19}
Dan Cristofaro-Gardiner, \emph{Symplectic embeddings from concave toric domains
  into convex ones}, J. Differential Geom. \textbf{112} (2019), no.~2,
  199--232, With an appendix by Cristofaro-Gardiner and Keon Choi. \MR{3960266}

\bibitem{C-GHS21}
Dan Cristofaro-Gardiner, Richard Hind, and Kyler Siegel, \emph{Higher
  symplectic capacities and the stabilized embedding problem for integral
  ellipsoids}, arXiv preprint arXiv:2102.07895.

\bibitem{C-GFS17}
Daniel Cristofaro-Gardiner, David Frenkel, and Felix Schlenk, \emph{Symplectic
  embeddings of four-dimensional ellipsoids into integral polydiscs}, Algebr.
  Geom. Topol. \textbf{17} (2017), no.~2, 1189--1260. \MR{3623687}

\bibitem{EH90}
Ivar Ekeland and Helmut Hofer, \emph{Symplectic topology and {H}amiltonian
  dynamics. {II}}, Math. Z. \textbf{203} (1990), no.~4, 553--567. \MR{1044064}

\bibitem{Eli91}
Yakov Eliashberg, \emph{New invariants of open symplectic and contact
  manifolds}, J. Amer. Math. Soc. \textbf{4} (1991), no.~3, 513--520.
  \MR{1102580}

\bibitem{EGH00}
Yakov Eliashberg, Alexander Givental, and Helmut Hofer, \emph{Introduction to
  symplectic field theory}, Geom. Funct. Anal. (2000), no.~Special Volume, Part
  II, 560--673, GAFA 2000 (Tel Aviv, 1999). \MR{1826267}

\bibitem{EGM18}
Michael Entov, Yaniv Ganor, and Cedric Membrez, \emph{Lagrangian isotopies and
  symplectic function theory}, Comment. Math. Helv. \textbf{93} (2018), no.~4,
  829--882. \MR{3880228}

\bibitem{GH18}
Jean Gutt and Michael Hutchings, \emph{Symplectic capacities from positive
  {$S^1$}-equivariant symplectic homology}, Algebr. Geom. Topol. \textbf{18}
  (2018), no.~6, 3537--3600. \MR{3868228}

\bibitem{GU19}
Jean Gutt and Michael Usher, \emph{Symplectically knotted codimension-zero
  embeddings of domains in {$\Bbb{R}^4$}}, Duke Math. J. \textbf{168} (2019),
  no.~12, 2299--2363. \MR{3999447}

\bibitem{HK18-2}
Richard Hind and Ely Kerman, \emph{{$J$}-holomorphic cylinders between
  ellipsoids in dimension four}, J. Symplectic Geom. \textbf{18} (2020), no.~5,
  1221--1245. \MR{4174300}

\bibitem{HO19}
Richard Hind and Emmanuel Opshtein, \emph{Squeezing {L}agrangian tori in
  dimension 4}, Comment. Math. Helv. \textbf{95} (2020), no.~3, 535--567.
  \MR{4152624}

\bibitem{hindzhang}
Richard Hind and Jun Zhang, \emph{The shape invariant of symplectic
  ellipsoids}, arXiv preprint arXiv: 2010.02185.

\bibitem{Hof06}
Helmut Hofer, \emph{A general {F}redholm theory and applications}, Int. Press,
  Somerville, MA, 2006. \MR{2459290}

\bibitem{Hut11}
Michael Hutchings, \emph{Quantitative embedded contact homology}, J.
  Differential Geom. \textbf{88} (2011), no.~2, 231--266. \MR{2838266}

\bibitem{McD91}
Dusa McDuff, \emph{Blow ups and symplectic embeddings in dimension {$4$}},
  Topology \textbf{30} (1991), no.~3, 409--421. \MR{1113685}

\bibitem{McD09}
\bysame, \emph{Symplectic embeddings of 4-dimensional ellipsoids}, J. Topol.
  \textbf{2} (2009), no.~1, 1--22. \MR{2499436}

\bibitem{McD11}
\bysame, \emph{The {H}ofer conjecture on embedding symplectic ellipsoids}, J.
  Differential Geom. \textbf{88} (2011), no.~3, 519--532. \MR{2844441}

\bibitem{McD18}
\bysame, \emph{A remark on the stabilized symplectic embedding problem for
  ellipsoids}, Eur. J. Math. \textbf{4} (2018), no.~1, 356--371. \MR{3782228}

\bibitem{MS12}
Dusa McDuff and Felix Schlenk, \emph{The embedding capacity of 4-dimensional
  symplectic ellipsoids}, Ann. of Math. (2) \textbf{175} (2012), no.~3,
  1191--1282. \MR{2912705}

\bibitem{Mul19}
Stefan M\"{u}ller, \emph{{$C^0$}-characterization of symplectic and contact
  embeddings and {L}agrangian rigidity}, Internat. J. Math. \textbf{30} (2019),
  no.~9, 1950035, 48. \MR{3995450}

\bibitem{RS93}
Joel Robbin and Dietmar Salamon, \emph{The {M}aslov index for paths}, Topology
  \textbf{32} (1993), no.~4, 827--844. \MR{1241874}

\bibitem{RZ20}
Daniel Rosen and Jun Zhang, \emph{Relative growth rate and contact
  {B}anach-{M}azur distance}, arXiv preprint arXiv:2001.05094.

\bibitem{Sch-book}
Felix Schlenk, \emph{Embedding problems in symplectic geometry}, De Gruyter
  Expositions in Mathematics, vol.~40, Walter de Gruyter GmbH \& Co. KG,
  Berlin, 2005. \MR{2147307}

\bibitem{STV18}
Egor Shelukhin, Dmitry Tonkonog, and Renato Vianna, \emph{Geometry of
  symplectic flux and lagrangian torus fibrations}, arXiv preprint
  arXiv:1804.02044.

\bibitem{siefring}
Richard Siefring, \emph{Intersection theory of punctured pseudoholomorphic
  curves}, Geom. Topol. \textbf{15} (2011), no.~4, 2351--2457. \MR{2862160}

\bibitem{Sieip}
Kyler Siegel, \emph{Computing higher symplectic capacities {I}{I}}, In
  progress.

\bibitem{Sik89}
Jean-Claude Sikorav, \emph{Rigidit\'{e} symplectique dans le cotangent de
  {$T^n$}}, Duke Math. J. \textbf{59} (1989), no.~3, 759--763. \MR{1046748}

\bibitem{SYZ96}
Andrew Strominger, Shing-Tung Yau, and Eric Zaslow, \emph{Mirror symmetry is
  {$T$}-duality}, Nuclear Phys. B \textbf{479} (1996), no.~1-2, 243--259.
  \MR{1429831}

\bibitem{Wen10}
Chris Wendl, \emph{Automatic transversality and orbifolds of punctured
  holomorphic curves in dimension four}, Comment. Math. Helv. \textbf{85}
  (2010), no.~2, 347--407. \MR{2595183}

\end{thebibliography}
\vspace*{5mm}
\end{document}